\documentclass[11pt]{amsart}

\usepackage{amssymb}   

\usepackage[all]{xy}
\usepackage[draft=false]{hyperref}
\usepackage{ae}

\usepackage{color}
\usepackage{epsfig}

\renewcommand\AA{\mathbb{A}}
\newcommand\PP{\mathbb{P}}
\newcommand\FF{\mathbb{F}}
\newcommand\FFqbar{{\overline{\FF}_q}}
\newcommand\ZZ{\mathbb{Z}}
\newcommand\QQ{\mathbb{Q}}
\newcommand\OO{\mathcal{O}}
\newcommand\RR{\mathbb{R}}
\newcommand\CC{\mathbb{C}}
\newcommand\R{\mathfrak{R}}
\newcommand\xx{\mathbf{x}}
\newcommand\Kbar{{\overline{K}}}
\newcommand\Pone{\PP^1}
\newcommand\Ptwo{\PP^2}
\newcommand\Pthree{\PP^3}

\newcommand\rto{\dashrightarrow}

\newcommand\lin[2]{{\langle{#1},{#2}\rangle}}

\DeclareMathOperator\Hom{Hom}
\DeclareMathOperator\Gal{Gal}

\DeclareMathOperator\id{id}

\newtheorem{theorem}{Theorem}
\newtheorem{lemma}[theorem]{Lemma}
\newtheorem{corollary}[theorem]{Corollary}
\newtheorem{proposition}[theorem]{Proposition}

\theoremstyle{definition}
\newtheorem{definition}[theorem]{Definition}
\newtheorem{conjecture}[theorem]{Conjecture}

\newtheorem{remark}[theorem]{Remark}
\newtheorem{example}[theorem]{Example}
\newtheorem{question}[theorem]{Question}

\newtheorem{exercise}[theorem]{Exercise}
\newtheorem{say}[theorem]{}

\numberwithin{theorem}{section}

\newcommand{\spec}[0]{\operatorname{Spec}}
\newcommand{\pic}[0]{\operatorname{Pic}}

\newcommand{\mult}[0]{\operatorname{mult}}
\def\into{\DOTSB\lhook\joinrel\to}
\newcommand{\smooth}[0]{\operatorname{Smoothing}}
\newcommand{\comb}[0]{\operatorname{Comb}}

\begin{document}

\title[Rational curves on 
cubic surfaces]{Looking for rational curves on \\
cubic hypersurfaces}

\author[J\'anos Koll\'ar]{J\'anos Koll\'ar\\
\\
notes  by  Ulrich Derenthal}

\date{G\"ottingen, 25--29 June 2007}

\maketitle
\today


\section{Introduction}

The aim of these lectures is to study
 rational points and rational curves on varieties, mainly
 over finite fields $\FF_q$. We concentrate on 
hypersurfaces  $X^n$ of degree $\leq n+1$ in  $\PP^{n+1}$,
especially on 
cubic hypersurfaces.

The theorem of Chevalley--Warning (cf.\ Esnault's
  lectures) guarantees  rational points on low degree
 hypersurfaces over finite fields. That is, if
   $X \subset \PP^{n+1}$ is a hypersurface of degree $ \le n+1$, then
  $X(\FF_q) \ne \emptyset$.

In particular, every cubic hypersurface of dimension $\geq 2$
defined over a finite field contains a
 rational point, but we would like to say more.
\begin{enumerate}
\item[$\bullet$]  Which cubic 
hypersurfaces contain more than one rational point? 
\item[$\bullet$]
Which cubic
hypersurfaces contain rational curves?
\item[$\bullet$]
Which cubic
hypersurfaces contain many rational curves?
\end{enumerate}

  Note that there
can be rational curves on $X$  even if $X$ has a unique $\FF_q$-point.
Indeed, $f:\Pone \to X$ 
 could map all $q+1$ points of $\Pone(\FF_q)$ to the same point
in $X(\FF_q)$, even if $f$ is not constant.

So what does it mean for a variety to 
contain many rational curves?
As an example, let us look at $\CC\PP^2$.
We know that through any 2 points there is a line, through any
5 points there is a conic, and so on. 
So we might say that a variety $X_K$ 
contains many rational curves if through any number of
points $p_1,\dots, p_n\in X(K)$ there is a rational curve
defined over $K$. 

However, we are in trouble over finite fields.
A smooth rational curve over $\FF_q$ has only $q+1$ points, so it can
never pass through more than $q+1$ points in $X(\FF_q)$. 
Thus, for cubic hypersurfaces, the following result,
proved in  Section \ref{desc.sect},  appears to  be optimal:

\begin{theorem} \label{cubic.ext.thm}
Let $X\subset \PP^{n+1}$ be a smooth cubic hypersurface
over $\FF_q$. Assume that $n\geq 2$ and $q\geq 8$.
Then every map of sets $\phi:\Pone(\FF_q)\to X(\FF_q)$ can be extended
to a map of $\FF_q$-varieties  $\Phi: \Pone \to X$.
\end{theorem}

In fact, one could think of stronger versions as well.
A good way to formulate what it means for $X$
to contain many (rational and nonrational)  curves is the following:

\begin{conjecture}\cite{k-sz} \label{koll.conj}
 $X\subset \PP^{n+1}$ be a smooth 
hypersurface of degree $\leq n+1$ defined over a finite field $\FF_q$.
 Let $C$ be a smooth projective curve and $Z \subset C$ a zero-dimensional
  subscheme. Then  any morphism $\phi: Z \to X$ can be extended to $C$.
That is,  there is a morphism
  $\Phi:C \to X$ such that $\Phi|_{Z}=\phi$.
\end{conjecture}

More generally, this should hold for any separably rationally connected
variety $X$,
see \cite{k-sz}. We define this notion in
Section \ref{sep-rat.sec}.

The aim of these notes is to explore these and related questions, 
especially for
cubic hypersurfaces. The emphasis will be on presenting a variety of
methods, and we end up outlining the proof of
 two special  cases of the Conjecture.

\begin{theorem} Conjecture (\ref{koll.conj}) holds in the following two cases.
\begin{enumerate}
\item  \cite{k-sz} For arbitrary $X$, when $q$ is sufficiently large
(depending on $\dim X, g(C)$ and $\deg Z$),  and
\item for cubic hypersurfaces when  $q\geq 8$
and $Z$ contains only odd degree points.
\end{enumerate}
\end{theorem}

As a warm-up, let us prove the case when 
$X= \PP^n$. This  is essentially due to Lagrange.
The case of quadrics is already quite a bit harder,
see (\ref{conj.bir.inv}).

\begin{example}[Polynomial interpolation]
  Over $\FF_q$, let $C$ be a smooth projective curve, $Z\subset C$
  a  zero-dimensional
  subscheme and $\phi: Z \to  \PP^n$
   a given map.

Fix a line bundle $L$ on $C$ such that
 $\deg L\geq |Z|+2g(C)-1$ and choose an isomorphism
${\mathcal O}_Z\cong L|_Z$.  Then $\phi$ can be given by
$n+1$ sections  $\phi_i\in H^0(Z, L|_Z)$.
From the exact sequence
$$
0\to L(-Z)\to L\to L|_Z\to 0
$$
we see that $H^0(C,L)\twoheadrightarrow H^0(Z, L|_Z)$.
Thus each $\phi_i$ lifts to
$\Phi_i\in H^0(C, L)$
giving the required extension $\Phi: C\to {\mathbb P}^n$.
\end{example}

\begin{say}[The plan of the lectures] 

In Section \ref{1pt.sec}, we study hypersurfaces with a unique point.
This is mostly for entertainment, though special examples
are frequently useful.

Then we prove that a smooth cubic hypersurface containing
a $K$-point is unirational over $K$. That is, there is a
dominant map $g:\PP^n\rto X$. This of course gives plenty of rational curves
on $X$ as  images of   rational curves
on $\PP^n$. Note however, that in general, 
$g:\PP^n(K)\rto X(K)$ is not onto.
(In fact, one expects the image to be very small, see \cite[Sec.VI.6]{manin}.)
Thus unirationality does not guarantee that there is a
rational curve through every $K$-point.

As a generalization of unirationality, the notion of
separably rationally connected varieties is
introduced in Section \ref{sep-rat.sec}.
This is the right class to study the existence of many rational curves.
Spaces parametrizing all rational curves on a  variety are
constructed in Section \ref{rcpar.sec} and their deformation theory
is studied in Section \ref{combs.sec}.

 The easy case of Conjecture \ref{koll.conj}
is
when $Z$ is a single $K$-point. Here a complete answer 
to the analogous question
is known
over $\RR$ or $\QQ_p$. Over $\FF_q$, the Lang-Weil estimates
give a positive answer for $q$ large enough; this is
reviewed in Section \ref{lw.sec}.

The first really hard case of (\ref{koll.conj}) is when
   $C = \Pone$ and $Z= \{0, \infty\}$.  The geometric question is: given $X$
  with $p,p' \in X(\FF_q)$, is there a rational curve defined over $\FF_q$
  passing through $p,p'$? We see in Section \ref{2pt.sec}
 that this is much harder
than the 1-point case since it is related to Lefschetz-type results on
the fundamental groups of open subvarieties.
We use this connection to settle the case 
  for $q$ large enough and $p,p'$ in general position.

Finally, in Section \ref{desc.sect} we use the previous result and the
``third intersection point map'' to prove Theorem \ref{cubic.ext.thm}.
\end{say}

\begin{remark}
The first indication that the 2-point case of (\ref{koll.conj})
is harder than the 1-point case is the different
behavior over
 $\RR$. Consider 
the cubic surface $S$ defined by the affine equation $y^2+z^2=x^3-x$.
Then $S(\RR)$ 
  has two components (a compact and an infinite part).
  \begin{itemize}
  \item If $p,p'$ lie in different components, there is  no rational
    curve over $\RR$ through $p,p'$, since $\RR\Pone$ is connected.
  \item If $p,p'$ lie in the same component, there is no topological
    obstruction. In fact, in this case an $\RR$-rational curve through $p,p'$
    always exists, see \cite[1.7]{k-loc}.
  \end{itemize}
\end{remark}

\section{Hypersurfaces with a unique point}\label{1pt.sec}

The first question has been answered  by Swinnerton-Dyer.
We state it in a seemingly
much sharpened form.

\begin{proposition}\label{prop:one_point}
Let $X$ be a  smooth 
 cubic hypersurface of dimension $\geq 2$
 defined over a  field 
$K$ with a unique $K$-point.
Then $\dim X=2$, $K=\FF_2$  and $X$ is unique up to
projective equivalence.
\end{proposition}

\begin{proof} 
We show in the next section that $X$ is unirational.
Hence, if $K$ is infinite, then $X$ has infinitely many $K$-points.
So this is really a question about finite fields.

If $\dim X\geq 3$ then  $|X(K)|\geq q+1$ by (\ref{lower.deg.CW}). Let us
   show next that there is no such surface over $\FF_q$ for $q \ge 3$.

 Assume to the contrary that $S$ contains exactly one rational point $x \in
  S(\FF_q)$. There are $q^3$ hyperplanes in $\Pthree$ over $\FF_q$ not passing
  through $x$. 

  The intersection of each hyperplane with $S$ is a curve of degree 3, which
  is either an irreducible cubic curve, or the union of a line and a conic, or
  the union of three lines. In the first case, the cubic curve contains a
  rational point (if $C$ contains a singular point, this points is defined
  over $\FF_q$; if $C$ is smooth, the Weil conjectures show that $|\#C(\FF_q)
  - (q+1)| \le 2\sqrt q$, so $\#C(\FF_q) = 0$ is impossible); in the second
  case, the line is rational and therefore contains rational points; in the
  third case, at least one line is rational unless all three lines are
  conjugate.

  By assumption,  for each hyperplane $H$ not passing through $x$,
the intersection  $S\cap H$
 does not contain rational points. Therefore, $S\cap H$ must be a union of
  three lines that are not defined over $\FF_q$, but conjugate over $\FFqbar$.

  This gives $3q^3$ lines on $S$. However, over an algebraically closed field,
  a cubic surface contains exactly 27 lines. For $q \ge 3$, we have arrived at
  a contradiction.

  Finally we construct the surface over $\FF_2$, without showing uniqueness.
That needs  a little more case analysis, see \cite[1.39]{ksc}.

  We may assume that the rational point is the origin of
  an affine space on which $S$ is given by the equation
  \[f(x,y,z) = z+Q(x,y,z)+C(x,y,z),\] with $Q$ (resp. $C$) homogeneous of
  degree 2 (resp. 3). If $C$ vanishes in $(x,y,z)$, then $S$ has a rational
  point $(x:y:z)$ on the hyperplane $\Ptwo$ at infinity. Since $C$ must not
  vanish in $(1,0,0)$, the cubic form $C$ must contain the term $x^3$, and
  similarly $y^3$, $z^3$. By considering $(1,1,0)$, we see that it must also
  contain $x^2y$ or $xy^2$, so without loss of generality, we may assume that
  it contains $x^2y$, and for similar reasons, we add the terms $y^2z$,
  $z^2x$. To ensure that $C$ does not vanish at $(1,1,1)$, we add the term
  $xyz$, giving
  \[C(x,y,z)=x^3+y^3+z^3+x^2y+y^2z+z^2x+xyz.\]
  Outside the hyperplane at infinity, we distinguish two cases:
  We see by considering a tangent plane $(z=0)$ that it  must intersect
    $S$ in  three conjugate
    lines. We conclude that $Q(x,y,z)=z(ax+by+cz)$
    for certain $a,b,c \in \FF_2$.
   For $z \ne 0$, we must ensure that $f$ does not vanish at the four
    points $(0,0,1)$, $(1,0,1)$, $(0,1,1)$, $(1,1,1)$, resulting in certain
    restrictions for $a,b,c$. These are satisfied by  $a=b=0$ and $c=1$.
  Therefore, $f(x,y,z) = z+z^2+C(x,y,z)$ with $C$ as above defines a cubic
  surface over $\FF_2$ that has exactly one $\FF_2$-rational point.

There are various ways to check that the cubic is smooth.
The direct computations are messy by hand but easy on a computer.
Alternatively, one can note that $S$ does contain 
$3\cdot 2^3+3=27$ lines
and singular cubics always have fewer than 27.
\end{proof}

\begin{remark}
  A variation of the  argument in the proof shows, without using the
  theorem of Chevalley--Warning, that a cubic surface $S$ defined over $\FF_q$
   must contain at least one rational point: 

  If $S$ does not contain a rational point, the intersection of $S$
  with any of the $q^3+q^2+q+1$ hyperplanes in $\Pthree$ consists of three
  conjugate lines, giving  $3(q^3+q^2+q+1) \ge 45>27$ lines,
a contradiction.
\end{remark}

\begin{exercise}\label{lower.deg.CW}  Using Chevalley--Warning, 
 show that for a hypersurface $X \subset \PP^{n+1}$ of degree 
$n+1-r$, the number of
  $\FF_q$-rational points is at least $|\PP^r(\FF_q)|=q^r+\cdots+q+1$.
\end{exercise}

\begin{question}\label{qu:one_point}
  Find more examples of hypersurfaces $X \subset \PP^{n+1}$ of degree at most
  $n+1$ with $\#X(\FF_q) = 1$.
\end{question}

\begin{example}[H.-C. Graf v. Bothmer]
  We construct hypersurfaces $X \subset \PP^{n+1}$ over $\FF_2$ containing
  exactly one rational point. 

  We start by constructing an affine equation. Note that the polynomial $f :=
  x_0\cdots x_{n+1}$ vanishes in every $\xx \in \FF_2^{n+2}$ except
  $(1,\dots,1)$, while $g := (x_0-1)\cdots(x_{n+1}-1)$ vanishes in every
  point except $(0,\dots,0)$. Therefore, the polynomial $h:=f+g+1$ vanishes
  only in $(0,\dots,0)$ and $(1,\dots,1)$.

  The only monomial of degree at least $n+2$ occurring in $f$ and $g$ is
  $x_0\cdots x_{n+1}$, while the constant term 1 occurs in $g$ but not in $f$.
  Therefore, $h$ is a polynomial of degree $n+1$ without constant term. We
  construct the homogeneous polynomial $H$ of degree $n+1$ from $h$ by
  replacing each monomial $x_{i_1}x_{i_2}\cdots x_{i_r}$ of degree $r \in \{1,
  \dots, n+1\}$ of $h$ with $i_1< \dots < i_r$ by $x_{i_1}^{k+1} x_{i_2}\cdots
  x_{i_r}$ of degree $n+1$ (where $k = n+1-r$). Since $a^k = a$ for any $k \ge
  1$ and $a \in \FF_2$, we have $h(\xx)=H(\xx)$ for any $\xx \in \FF_2^{n+2}$;
  the homogeneous polynomial $H$ vanishes exactly in $(0, \dots, 0)$ and $(1,
  \dots, 1)$.

  Therefore, $H$ defines a degree $n+1$ hypersurface  $\PP^{n+1}$
  containing exactly one $\FF_2$-rational point $(1:\dots:1)$.

  Using a computer, we can check for $n=2, 3, 4$ that $H$ defines a smooth
  hypersurface of dimension $n$. Note that for $n=2$, the resulting cubic
  surface is isomorphic to the one constructed in
  Proposition~\ref{prop:one_point}. For $n \ge 5$, it is unknown whether $H$
  defines a smooth variety.
\end{example}

\begin{example} Let $\alpha_1$ be a generator or $\FF_{q^m}/\FF_q$
with conjugates $\alpha_i$. It is easy to see that
$$
X(\alpha):=\Bigl(\prod_i(x_1+\alpha_ix_2+\cdots+\alpha_i^{m-1}x_m)=0\Bigr)
\subset \PP^m
$$
has a unique $\FF_q$-point at $(1:0:\cdots:0)$. 
$X(\alpha)$ has degree $m$, it is irreducible over
$\FF_q$ but over $\FF_{q^m}$
it is the union of $m$ planes.

Assume now that $q\leq m-1$. 
Note that $x_i^qx_j-x_ix_j^q$ is identically zero on $\PP^m(\FF_q)$.
Let $H$ be any homogeneous degree $m$ element of the ideal
generated by all the $x_i^qx_j-x_ix_j^q$. Then $H$ is also
identically zero on $\PP^m(\FF_q)$, thus
$$
X(\alpha,H):=\Bigl(\prod_i(x_1+\alpha_ix_2+\cdots+\alpha_i^{m-1}x_m)=H\Bigr)
\subset \PP^m
$$
also has a unique $\FF_q$-point at $(1:0:\cdots:0)$.

By computer it is again possible to find further examples of
smooth hypersurfaces with a unique point, but
the computations seem exceedingly lengthy for $m\geq 6$.
\end{example}

\begin{remark} Let $X\subset \PP^{n+1}$ be a smooth hypersurface of degree
$d$. Then the primitive middle Betti number is
$$
\frac{(d-1)^{n+2}+(-1)^n}{d}+1+(-1)^n\leq d^{n+1}.
$$
Thus by the Deligne-Weil estimates
$$
\bigl| \# X(\FF_q)-\# \PP^n(\FF_q)\bigr|\leq d^{n+1} q^{n/2}.
$$
Thus we get that for $d=n+1$, there are more then
$\# \PP^{n-1}(\FF_q)$ points in $X(\FF_q)$ as soon as
$q\geq (n+1)^{2+\frac2{n}}$. 

\end{remark}

\section{Unirationality}

\begin{definition}
  A variety $X$ of dimension $n$ defined over a field $K$ is called
  \emph{unirational} if there is a dominant map $\phi: \PP^n \rto X$,
  also defined over $K$.
\end{definition}

\begin{exercise} \cite[2.3]{k-uni}  Assume that 
there is a dominant map $\phi: \PP^N \rto X$ for some $N$.
Show that $X$ is unirational.
\end{exercise}

The following result was proved by Segre \cite{segre} in the case $n=2$,
by Manin \cite{manin} for arbitrary $n$ and general $X$
when $K$ is not a finite
field with ``too few'' elements, and in full generality
 by Koll\'ar \cite{k-uni}.

\begin{theorem}\label{thm:unirationl-rational_point}
  Let $K$ be an arbitrary field and  $X \subset \PP^{n+1}$ a smooth cubic
  hypersurface ($n \ge 2$). Then the following are equivalent:
  \begin{enumerate}
  \item\label{it:dominant_map} $X$ is unirational over $K$.
  \item\label{it:rational_point} $X(K) \ne \emptyset$.
  \end{enumerate}
\end{theorem}

{\it Proof.}
Let us start with the easy direction: \eqref{it:dominant_map} $\Rightarrow$
 \eqref{it:rational_point}. The proof of the other direction
will occupy the rest of the section.

  If $K$ is infinite, then  $K$-rational points in
  $\PP^n$ are Zariski-dense, so $\phi$ is defined on most of them, giving
  $K$-rational points of $X$ as their image.

  If $K$ is a finite field, $\phi$ might not be defined on any $K$-rational
  point. Here, the result is a special case of the following.

\begin{lemma}[Nishimura]
  Given a smooth variety $Y$ defined over $K$ with $Y(K) \ne \emptyset$ and a
  rational map $\phi: Y \rto Z$ with $Z$ proper, we have $Z(K) \ne \emptyset$.
\end{lemma}

\begin{proof}[Proof (after E. Szab\'o)]
  We proceed by induction on the dimension of $Y$. If $\dim Y = 0$, the
  result is clear. If $\dim Y = d > 0$, we extend $\phi: Y \rto Z$ to $\phi':
  Y' \rto Z$, where $Y'$ is the blow-up of $Y$ in $p \in Y(K)$. Since a
  rational map is defined outside a closed subset of codimension at least 2,
  we can restrict $\phi'$ to the exceptional divisor, which is isomorphic to
  $\PP^{d-1}$. This restriction is a map satisfying the induction hypothesis.
  Therefore, $X(K) \ne \emptyset$.
\end{proof}

\begin{say}[Third intersection point map]\label{tipm.say}

Let $C\subset \Ptwo$ be a smooth cubic curve.
For $p, p'\in C$  the line $\lin{p}{p'}$ through them
 intersects $C$ in
a unique third point, denote it by $\phi(p,p')$. 
The resulting morphism $\phi: C\times C\to C$
is, up to a choice of the origin and a sign, the group law
on the elliptic curve $C$.

For an arbitrary cubic hypersurface $X$ defined over a  field $K$, 
we can construct the analogous
rational map $\phi: X \times X \rto X$ as follows.
  If $p \ne p'$ and if the line $\lin{p}{p'}$
 does not lie completely in $X$, it intersects $X$ in
a unique third point $\phi(p,p')$. 
If $X_{\bar K}$ is irreducible, 
this defines $\phi$ on an open subset of $X
\times X$.

It is very tempting to believe that out of $\phi$
one can get an (at least  birational) group law on $X$.
This is, unfortunately, not at all the case. 
The book \cite{manin} gives a detailed exploration of 
this direction.

We use $\phi$  to
 obtain a dominant map from a projective space to $X$, relying on
 two basic ideas:

\begin{itemize}
\item Assume that $Y_1, Y_2 \subset X$ are rational subvarieties  such that 
$\dim Y_1+\dim Y_2\geq \dim X$. Then, if $Y_1, Y_2$ are in
``general position,'' 
the
  restriction of $\phi$ gives  a dominant map $Y_1 \times Y_2
  \rto X$. Thus $X$ is unirational since  $Y_1 \times Y_2$ is birational 
to a projective space.
\item How can we find rational subvarieties of $X$?
Pick a rational point
  $p \in X(K)$ and let   $Y_p$ be the intersection of $X$ with 
the tangent hyperplane 
$T_p$ of $X$ in $p$. Note that  $Y_p$
  is  a cubic hypersurface of dimension $n-1$ with a singularity
  at $p$. 

If $p$ is  in ``general position,'' then 
$Y_p$ is irreducible and not a cone. Thus $\pi: Y_p
  \rto \PP^{n-1}$, 
 the projection from $p$,   is birational and so $Y_p$ is rational.
\end{itemize}

From this we conclude that if $X(K)$ has at lest 2 points
in ``general position,'' then $X$ is unirational.
In order to prove unirationality, one needs to
understand the precise meaning of the above 
``general position'' restrictions, and then figure out what to do
if there are no points in ``general position.''
This is especially a problem over finite fields.
\end{say}

\begin{example} \cite{hirsch} Check that over $\FF_2, \FF_4$ and $\FF_{16}$
all points of $(x_0^3+x_1^3+x_2^3+x_3^3=0)$ lie on a line.
In particular, the curves $Y_p$ are  reducible whenever
$p$  is over $\FF_{16}$. Thus there are no points in ``general position.''
\end{example}

\begin{say}[End of the proof of (\ref{thm:unirationl-rational_point})]
\label{pf.of.3.3}
In order to prove (\ref{thm:unirationl-rational_point}),
we describe 3 constructions, working in  increasing generality.

(\ref{pf.of.3.3}.1) 
 Pick $p \in X(K)$. If $Y_p$ is irreducible and not a cone, 
 then  $Y_p$  is birational over $K$ to
  $\PP^{n-1}$.  This gives more $K$-rational points on $Y_p$. We pick $p'
  \in Y_p(K) \subset X(K)$, and if we are lucky again, $Y_{p'}$ 
 is also birational over $K$ to $\PP^{n-1}$. This results in
  \[\Phi_{1,p,p'}:\PP^{2n-2} \stackrel{bir}{\rto} Y_p \times Y_{p'} \rto X,\] 
where the first map is
  birational and the second map is dominant.

This construction works when $K$ is infinite and 
$Y_p$ is irreducible and not a cone.

(\ref{pf.of.3.3}.2) 
 Over  $K$, it might be impossible to find 
$p\in X(K)$ such that $Y_p$ is irreducible.  Here we try to give ourselves
a little more room by passing to a quadratic field extension
 and then coming back to
$K$ using the third intersection point map $\phi$.

  Given $p \in X(K)$,  a  line through $p$ can intersect $X$ in two
  conjugate points $s,s'$ defined over a quadratic field extension
$K'/K$. If $Y_s, Y_{s'}$ (the intersections of $X$ with the
  tangent hyperplanes in $s$ resp. $s'$) are birational to $\PP^{n-1}$ over
  $K'$, consider the map \[\Phi_{1,s,s'}:Y_s \times Y_{s'} \rto X.\]
So far $\Phi_{1,s,s'}$ is defined over $K'$. 

Note however that $Y_s, Y_{s'}$ are conjugates of each other by the
Galois involution of $K'/K$. Furthermore, if $z\in Y_s$ and 
$\bar z\in Y_{s'}$ is its conjugate then the line 
 $\phi(z,\overline z)$ is defined over
  $K$. Indeed,  the Galois action interchanges $z, \overline z$ hence
 the line
  $\lin{z}{\overline z}$  is Galois invariant,  hence the third
  intersection point $\Phi_{1,s,s'}(z,\overline z)$ is defined over $K$. 
  
That is, the involution  $(z_1, z_2)\mapsto (\bar z_2, \bar z_1)$
makes $Y_s \times Y_{s'}$ into a $K$-variety and
 $\Phi_{1,s,s'}$ then becomes a $K$-morphism.
Thus we obtain
  a dominant map
  \[\Phi_{2,p,L}:\R_{K'/K}\PP^{n-1} \rto X\] where
  $\R_{K'/K}\PP^{n-1}$ is the Weil restriction of $\PP^{n-1}$ (cf.
  Example~\ref{ex:weil_restriction}).

This construction works when $K$ is infinite, even if  
$Y_p$ is reducible or a cone.
However, over a finite field, it may be impossible to find  a suitable line
$L$. 

As a last try, if none of the lines work, let's work with all lines
together!

(\ref{pf.of.3.3}.3)
  Consider the \emph{universal line} through $p$ instead of
  choosing a specific line. That is, we are working with all lines at once.  
To see what this means, choose  an affine equation such that
$p$ is at the origin:
$$
L(x_1,\dots,x_{n+1})+Q(x_1,\dots,x_{n+1})+C(x_1,\dots,x_{n+1})=0,
$$
where $L$ is linear, $Q$ is quadratic and $C$ is cubic.
The universal line is given by  $(m_1t, \dots, m_{n}t, t)$
where the $m_i$ are algebraically independent over $K$
and the quadratic formula gives the points $s,s'$ at
$$
t=\frac{-Q(m_1, \dots, m_{n}, 1)\pm\sqrt{D(m_1, \dots, m_{n}, 1)}}
{2C(m_1, \dots, m_{n-1}, 1)},
$$
where $D=Q^2-4LC$ is the discriminant.

Instead of working with just one pair $Y_s, Y_{s'}$, we work with the
universal family of them defined over the field
$$
K\Bigl(m_1, \dots,  m_{n}, \sqrt{D(m_1, \dots, m_{n}, 1)}\Bigr)
$$

It does not matter any longer that $Y_s$ may be reducible
for every $m_1, \dots,  m_{n}\in K$ since we are working with all the $Y_s$
together and the generic $Y_s$ is irreducible and not a cone.

Thus we get  a map 
$$
\Phi_{3,p}: \PP^n \times \PP^{n-1}
  \times \PP^{n-1}  \stackrel{bir}{\rto}
\R_{K\bigl(m_1, \dots,  m_{n}, \sqrt{D}\bigr)/
K\bigl(m_1, \dots,  m_{n}\bigr)}\PP^{n-1} \rto X.
$$

The last step is the following  observation.
Unirationality of $X_K$ changes if we extend $K$. However,
once we have a $K$-map $\PP^{3n-2} \rto X$, its dominance can be
checked after any field extension. Since
$\Phi_{3,p}$ incorporates all 
$\Phi_{2,p,L}$, we see that the $K$-map  $\Phi_{3,p}$
is dominant if $\Phi_{2,p,L}$ is dominant for some $\bar K$-line $L$.

 Thus  we can check dominance over
  the algebraic closure of $K$,
 where the techniques of the previous cases work.

There are a few remaining points to settle
(mainly that $Y_p$ is irreducible and not a cone
for general $p\in X(\bar K)$ and that
$\Phi_{1,p,p'}$ is dominant for general $p, p'\in X(\bar K)$).
These are left to the reader.
For more details, see
\cite[Section~2]{k-uni}.
\end{say}

\begin{example}\label{ex:weil_restriction}
  We give an explicit example of the construction of the Weil restriction. 
The aim of Weil restriction is to start with a finite field extension
$L/K$ and an $L$-variety
$X$ and construct in a natural way a $K$-variety 
$\R_{L/K}X$ such that  $X(L)=\bigl(\R_{L/K}X\bigr)(K)$.

As a good example, assume that the characteristic is $\neq 2$ and 
  let $L = K(\sqrt a)$ be a quadratic field extension
with  $G:=\Gal(L/K) =\{\id , \sigma\}$.
Let $X$ be an $L$-variety and $ X^{\sigma}$ its conjugate over $K$.

Then $X\times X^{\sigma}$ is an $L$-variety. We can define a
$G$-action on it by
$$
\sigma: (x_1,  x_2^{\sigma})\mapsto (x_2,  x_1^{\sigma}).
$$
This makes $X\times X^{\sigma}$ into the $K$-variety
$\R_{L/K}X$.

We explicitly construct
  $\R_{L/K}\Pone$, which is all one needs for the surface
case of (\ref{thm:unirationl-rational_point}).

Take the product of two copies of $\Pone$ with the $G$-action 
 $$
((s_1:t_1),(s_2^{\sigma}:t_2^{\sigma}))\mapsto
 ((s_2:t_2),(s_1^{\sigma}:t_1^{\sigma})).
$$
  Sections of $\OO(1,1)$ invariant under $G$ are \[u_1:=s_1s_2^{\sigma}, \quad
  u_2:=t_1t_2^{\sigma},\quad u_3:=s_1t_2^{\sigma}+s_2^{\sigma}t_1,\quad
 u_4:=\frac{1}{\sqrt
    a}(s_1t_2^{\sigma}-s_2^{\sigma}t_1).\] These sections satisfy
  $u_3^2-au_4^2=4u_1u_2$, and in fact, this equation defines $\R_{L/K}\Pone$
  as a subvariety of $\Pthree$ over $K$. 

Thus  $\R_{L/K}\Pone$ is a quadric surface with $K$-points
(e.g., $(1:0:0:0)$), hence rational over $K$.

  Let us check that  $(\R_{L/K}\Pone)(K) = \Pone(L)$. Explicitly, one
  direction of this correspondence is as follows. Given $(x_1+\sqrt a
  x_2:y_1+\sqrt a y_2) \in \Pone(L)$ with $x_1,x_2,y_1,y_2\in K$, we get the
  $G$-invariant point 
\[\bigl((x_1+\sqrt a x_2:y_1+\sqrt a y_2),(x_1-\sqrt a
  x_2:y_1-\sqrt a y_2)\bigr) \in ((X \times X^{\sigma})(L))^G.\] 
From this, we compute
  \[u_1 = x_1^2-ax_2^2,\quad u_2 = y_1^2-ay_2^2,\quad u_3=
  2(x_1y_1-ax_2y_2),\quad u_4=2(x_2y_1-x_1y_2).\] Then $(u_1:u_2:u_3:u_4) \in
  \R_{L/K}\Pone(K)$.

For the precise definitions and for
  more information, see \cite[Section~7.6]{bo-lu} or 
\cite[Definition~2.1]{k-uni}.

\end{example}

In order to illustrate the level of our ignorance about
unirationality, let me mention the following problem.

\begin{question}
  Over any field $K$, find an example of a smooth hypersurface $X \subset
  \PP^{n+1}$ with $\deg X \le n+1$ and $X(K) \ne \emptyset$ that is not
  unirational. So far, no such $X$ is known.
\end{question}

The following are 2 further  incarnations of the
third intersection point map.

\begin{exercise} Let $X^n$ be an irreducible cubic
hypersurface. Show that $S^2X$ is birational to
$X\times \PP^n$, where $S^2X$ denotes the symmetric square of $X$,
that is, $X\times X$ modulo the involution $(x,x')\mapsto (x',x)$.
\end{exercise}

\begin{exercise} Let $X^n$ be an irreducible cubic
hypersurface defined over $K$ and $L/K$ any
quadratic extension. Show that there
is a map  $\R_{L/K}X\rto X$.
\end{exercise}

\section{Separably rationally connected varieties}
\label{sep-rat.sec}

Before we start looking for rational curves on varieties over
finite fields, we should contemplate which varieties contain
plenty of rational curves over an algebraically closed
field. There are various possible ways of defining what we mean by
lots of rational curves, here are some of them.

\begin{say}\label{src.cond}
Let $X$ be a smooth projective variety over an algebraically closed
field $\Kbar$.
Consider the following conditions:
\begin{enumerate}
\item\label{it:two_points} For any given $x,x' \in X$, there is $f : \Pone \to X$ such that $f(0)=x,
  f(\infty) = x'$.
\item\label{it:m_points} For any given $x_1, \dots, x_m \in X$, there is $f: \Pone \to X$ such that
  $\{x_1, \dots, x_m\} \subset f(\Pone)$.
\item\label{it:tangent} Let 
 $Z \subset \Pone$ be a zero-dimensional
  subscheme and $f_Z: Z \to X$ a morphism. Then  $f_Z$
 can be extended to $f:\Pone \to X$.
\item\label{it:tangent_positive} Conditions  \eqref{it:tangent} holds, and
  furthermore $f^*T_X(-Z)$ is ample.  
(That is, $f^*T_X$ is a sum of line
  bundles each of degree at least $|Z|+1$.)
\item\label{it:positive} There is $f : \Pone \to X$ such that $f^*T_X$ is
  ample.
\end{enumerate}
\end{say}

\begin{theorem} \cite{kmm2}, \cite[Sec.4.3]{rc-book} 
\label{src.main.thm} Notation as above.
\begin{enumerate}
 \item  If  $\Kbar$ is an  uncountable field of
  characteristic $0$ then the  conditions
  \ref{src.cond}.\ref{it:two_points}--\ref{src.cond}.\ref{it:positive} are equivalent. 
\item For any  $\Kbar$, condition \ref{src.cond}.\ref{it:positive}
implies the others. 
\end{enumerate}
\end{theorem}

\begin{definition} Let $X$ be a smooth projective variety over a
field $K$. We say that $X$ is {\it separably rationally connected}
or {\it SRC} if the  conditions
  \ref{src.cond}.\ref{it:two_points}--\ref{src.cond}.\ref{it:positive} hold for $X_{\Kbar}$.
\end{definition}

\begin{remark} There are 2 reasons why  the  conditions
  \ref{src.cond}.\ref{it:two_points}--\ref{src.cond}.\ref{it:positive} are not always equivalent.

First, in positive characteristic, there are inseparably unirational
varieties. These also satisfy the  conditions
  \ref{src.cond}.\ref{it:two_points}--\ref{src.cond}.\ref{it:m_points}, but usually
not \ref{src.cond}.\ref{it:positive}. For instance, if $X$ is an
inseparably unirational surface of general type, then 
\ref{src.cond}.\ref{it:positive} fails. Such examples are given by
(resolutions of) a hypersurface of the form
$z^p=f(x,y)$ for $\deg f\gg 1$.

Second, over countable fields, 
it could happen that (\ref{src.cond}.\ref{it:two_points})
holds but 
$X$ has only countably many rational curves. In particular, 
the degree of the required $f$ depends on $x,x'$.
These examples are not easy to find, see \cite{bog-tsch}
for some over $\FFqbar$. It is not known if this can happen over
$\overline{\mathbb Q}$ or not.

  Over countable fields of characteristic $0$, we must require the existence
  of $f : \Pone \to X$ of \emph{bounded degree} in these conditions in order
  to obtain equivalence with \ref{src.cond}.\ref{it:positive}.
\end{remark}

\begin{example}
  Let $S \subset \Pthree$ be a cubic surface. Over the algebraic closure, $S$
  is the blow-up of $\Ptwo$ in six points. Considering $f$ mapping $\Pone$ to
  a line in $\Ptwo$ not passing through any of the six points, we see that $S$
  is separably rationally connected.
 
  More generally, any rational surface is separably rationally connected.

It is  not quite trivial to see that for any normal cubic surface $S$
that is  not a cone, there is a morphism to the smooth locus
$f:\Pone\to S^{ns}$ such that $f^*T_S$ is ample. 

Any normal cubic hypersurface is also separably rationally connected,
except cones over cubic curves.
To see this, take repeated general hyperplane sections until we get
a normal cubic surface $S\subset X$ which is not a cone. 
The normal bundle of $S$ in $X$ is ample, hence the
$f:\Pone\to S^{ns}$ found earlier also works for $X$.

In characteristic 0, any smooth hypersurface $X\subset \PP^{n+1}$
of degree $\leq n+1$ is SRC; see \cite[Sec.V.2]{rc-book}
for references and various stronger versions.
 Probably every normal hypersurface 
is also SRC, except for cones.

In positive characteristic the situation is more complicated.
A general hypersurface of degree $\leq n+1$ is SRC,
but it is not known that every smooth hypersurface
of degree $\leq n+1$ is SRC. There are some
 mildly singular hypersurfaces which are not SRC, see
\cite[Sec.V.5]{rc-book}.
\end{example}

\begin{say}[Effective bounds for hypersurfaces]\label{low.deg.rtl.curves}

Let $X\subset \PP^{n+1}$ be a smooth SRC hypersurface
over $\bar K$. Then (\ref{src.main.thm}) implies that
there are rational curves through any point or any 2 points.
Here we consider effective bounds for the degrees of such curves.

First, if $\deg X <n+1$ then through every point there are lines.
For a general point, the general line is also free (cf.\ (\ref{free.vf.maps})).

If $\deg X=n+1$ then there are no lines through a general
point, but usually there are conics.  However, on
a cubic surface  there are no irreducible
conics through  an Eckart point $p$. (See (\ref{cubic.surf.1pt.exmp}) for
the definition and details.)

My guess is that in all cases there are free twisted cubics through any point,
but this may be difficult to check.
I don't know any reasonable effective upper bound.

For 2 general points $x, x'\in X$, there is an irreducible
rational curve  of degree
$\leq n(n+1)/(n+2-\deg X)$ by \cite{kmm1}. The optimal result should be closer
to $n+1$, but this is not known. Very little is known
about non-general points.
\end{say}

Next we show that (\ref{koll.conj}) depends only on the birational class
of $X$. The proof also shows that (\ref{src.cond}.\ref{it:tangent}) is also
a birational property.

\begin{proposition} \label{conj.bir.inv} Let $K$ be a field and $X, X'$ 
 smooth projective $K$-varieties
which are birational to each other.
Then, if (\ref{koll.conj}) holds for $X$, it also holds for $X'$.
\end{proposition}

Proof.  Assume for notational simplicity that $K$ is perfect.
Fix embeddings $X\subset \PP^N$ and $X'\subset \PP^M$ and represent
the birational maps $\phi:X\rto X'$
and $\phi^{-1}:X'\rto X$ with polynomial coordinate functions.

Given $f'_Z:Z\to X'$, we  construct a thickening
 $Z\subset Z_t\subset C$ and $f_{Z_t}:Z_t\to X$
such that if $f:C\to X$ extends $f_{Z_t}$  then
$f':=\phi\circ f$ extends $f'_Z$.

Pick a point $p\in Z$ and let $\hat{\OO}_p\cong L[[t]]$ be its complete
local ring where $L=K(p)$. 
Then the corresponding component of $Z$ is $\spec_K L[[t]]/(t^m)$
for some $m\geq 1$.

By choosing suitable local coordinates
at $f'_Z(p)\in X'$, we can define its completion
by equations
$$
y_{n+i}=G_i(y_1,\dots,y_n)\quad\mbox{where $G_i\in L[[y_1,\dots,y_n]]$.}
$$
Thus $f'_Z$ is given by its coordinate functions
$$
\bar y_1(t),\dots, \bar y_n(t), \dots\in L[[t]]/(t^m).
$$
The polynomials $\bar y_i$ for $i=1,\dots, n$ can be lifted to
$y_i(t)\in L[[t]]$ arbitrarily. 
These then determine liftings $y_{n+i}(t)=G_i(y_1(t),\dots,y_n(t))$
giving a map $F':\spec_K L[[t]]\to X'$. 
In particular, we can choose a
lifting such that  $\phi^{-1}$ is a local isomorphism at the 
image of the generic point of $ \spec_K L[[t]]$.
Thus $\phi^{-1}\circ F'$ and
 $\phi\circ\phi^{-1}\circ F'$ are both defined and
$\phi\circ\phi^{-1}\circ F'=F'$.

Using the polynomial representations for  $\phi, \phi^{-1}$, write
$$
\begin{array}{rcl}
\phi^{-1}\circ (1,y_1(t),\dots,  y_n(t), \dots)&=&(x_0(t),\dots, x_n(t),\dots),
\quad\mbox{and}\\
\phi\circ (x_0(t),\dots, x_n(t),\dots)&=&(z_0(t),\dots,  z_n(t), \dots).
\end{array}
$$
Note that $(z_0(t),\dots,  z_n(t), \dots)$ and
$(1, y_1(t),\dots,  y_n(t), \dots)$ give the same map
$F':\spec_K L[[t]]\to X'$, but we map to projective space. 
Thus all we can say is that 
$$
y_i(t)= z_i(t)/z_0(t) \quad \forall\ i\geq 1.
$$
Assume now that we have $(x^*_0(t),\dots, x^*_n(t),\dots)$
and the corresponding
$$
\phi\circ (x^*_0(t),\dots, x^*_n(t),\dots)=(z^*_0(t),\dots,  z_n^*(t), \dots).
$$
such that 
 $$
x_i(t)\equiv x^*_i(t)\mod t^{s} \quad \forall\ i.
$$
Then also
 $$
z_i(t)\equiv z^*_i(t)\mod t^{s} \quad \forall\ i.
$$
In particular, if $s>r:=\mult_0 z_0(t)$, then
$\mult_0 z^*_0(t)=\mult_0 z_0(t)$ and so
$$
y^*_i(t):= \frac{z^*_i(t)}{z^*_0(t)}\equiv y_i(t)\mod \bigl( t^{s-r}\bigr).
$$
That is, if $F^*:\spec_K L[[t]]\to X$ agrees with
$\phi^{-1}\circ F'$ up to order $s= r+m$
then $\phi\circ F^*$ agrees with
$ F'$ up to order $m= s-r$.

We apply this to every point in $Z$ to obtain the thickening
 $Z\subset Z_t\subset C$ and $f_{Z_t}:Z_t\to X$
as required.\qed

\section{Spaces of rational curves}\label{rcpar.sec}

Assume that $X$ is defined over a non-closed field $K$ and is separably
rationally connected. Then $X$ contains lots of rational curves
over $\Kbar$, but what about rational curves over $K$? We are particularly
interested in the cases when $K$ is one of $\FF_q$, $\QQ_p$ or  $\RR$.

\begin{say}[Spaces of rational curves]
\label{sp.of.rc.say}
Let $X$ be any variety. Subvarieties or subschemes of $X$ come
in families, parametrized by the Chow variety or the Hilbert scheme.
For rational curves in $X$, the easiest  to describe is the
space of maps $\Hom(\Pone,X)$.

  Pick an embedding  $X \subset \PP^N$ and let $F_i$ be 
  homogeneous equations of $X$.

  Any map $f : \Pone \to \PP^N$ of fixed degree $d$ is given by $N+1$
  homogeneous polynomials $(f_0(s,t), \dots, f_N(s,t))$ of degree $d$ in two
  variables $s,t$ (up to scaling of these polynomials). Using the coefficients
  of $f_0, \dots, f_N$, we can regard $f$ as a point in $\PP^{(N+1)(d+1)-1}$.

  We have $f(\Pone) \subset X$ if and only if the polynomials $F_i(f_0(s,t),
  \dots, f_N(s,t))$ are identically zero. Each $F_i$  gives 
$d\cdot\deg F_i+1$ equations of degree
  $\deg F_i$ in the coefficients of $f_0, \dots, f_N$. 

If $f_0, \dots, f_N$ have a common zero, then we get only a lower degree map.
We do not  count these in $\Hom_d(\Pone,X)$.
By contrast we allow the possibility  that $f\in \Hom_d(\Pone,X)$
 is not an embedding but a
degree $e$ map onto a degree $d/e$ rational curve in $X$.
These maps clearly cause some trouble but, as it turns out,  it would be
technically very inconvenient to exclude them from the beginning.

Thus $\Hom_d(\Pone,X)$ is an open subset of a subvariety
of $\PP^{(N+1)(d+1)-1}$ defined by equations of degree 
$\leq \max_i\{\deg F_i\}$.

$\Hom(\Pone, X)$ is the disjoint union of the $\Hom_d(\Pone, X)$
for $d=1,2,\dots$.

  Therefore, finding a rational curve $f : \Pone \to X$   defined
  over $K$ is equivalent to finding $K$-points on 
  $\Hom(\Pone,X)$.

In a similar manner one can treat the space
 $\Hom(\Pone, X, 0\mapsto x)$ 
of those maps $f:\Pone\to X$ that satisfy $f(0)=x$
or  $\Hom(\Pone, X, 0 \mapsto x, \infty \mapsto x')$,
those maps $f:\Pone\to X$ that satisfy $f(0)=x$ and $f(\infty)=x'$.

\end{say}

\begin{say}[Free and very free maps]\label{free.vf.maps} In general, the local
structure of the spaces $\Hom(\Pone, X)$ can be very complicated,
but everything works nicely in certain important cases.

We say that $f:\Pone\to X$ is {\it free} if $f^*T_X$
is semi-positive, that is a direct sum of line bundles
of degree $\geq 0$. We see in (\ref{m0mbar.thm})  that
if $f$ is free then $\Hom(\Pone, X)$ and 
$\Hom(\Pone, X, 0\mapsto f(0))$ are both smooth at $[f]$.

We say that $f:\Pone\to X$ is {\it very free} if $f^*T_X$
is positive or ample, that is, a direct sum of line bundles
of degree $\geq 1$. 
This implies that $\Hom(\Pone, X, 0 \mapsto f(0), \infty \mapsto f(\infty))$
is also smooth at $[f]$.
\end{say}

\begin{remark} Over a nonclosed field $K$ there can be smooth
projective curves $C$ such that $C_{\bar K}\cong \Pone_{\bar K}$
but $C(K)=\emptyset$, thus $C$ is not birational to $\Pone_K$.
When we work with  $\Hom(\Pone, X)$, we definitely miss these
curves. There are various ways to remedy this problem,
but for us this is not important.

Over a finite field $K$, every rational curve is in fact
birational to $\Pone_K$, thus we do not miss anything.
\end{remark}

To get a feeling for these spaces,  let us see what we can  say about the
irreducible components of $\Hom_d(\Pone,X)$ for cubic surfaces.

\begin{example}\label{ex:cubic_curves}
  Let $S \subset \Pthree$ be a cubic surface defined over a non-closed field
  $K$. Consider $\Hom_d(\Pone, S)$ for low values of $d$.
  \begin{itemize}
  \item For $d=1$, over $\Kbar$, there are 27 lines on $S$, so
    $\Hom_1(\Pone,S)$ has 27 components which may be permuted by the action of
    the Galois group $G = \Gal(\Kbar/K)$.
  \item For $d=2$, over $\Kbar$, there are 27 one-dimensional families of
    conics, each obtained by intersecting $S$ with the pencil of  planes
 containing a line on $S$.  These  27 families again may have a non-trivial
    action of $G$.
  \item For $d=3$, over $\Kbar$, there are 72 two-dimensional families of
    twisted cubics on $S$ (corresponding to the 72 ways to map $S$ to $\Ptwo$
    by contracting 6 skew lines; the twisted cubics are preimages of lines in
    $\Ptwo$ not going through any of the six blown-up points). Again
there is no reason to assume that any of these 72 families is fixed
by $G$.

    However, there is exactly one two-dimensional family of plane rational 
cubic curves
    on $S$, obtained by intersecting  $S$ with planes tangent to the points on $S$
    outside the 27 lines. This family is defined over $K$ and is
geometrically irreducible.
  \end{itemize}

All this is not very surprising. A curve $C$ on $S$ determines
a line bundle $\OO_S(C)\in \pic(S)\cong \ZZ^7$, hence we see many
different families in a given degree because there 
are many
different line bundles of a given degree. It turns out that, 
for cubic surfaces, once we fix not just the degree but also
the line bundle $L=\OO_S(C)$, the resulting spaces
$\Hom_L(\Pone, X)$ are irreducible.

This, however, is a very special property of cubic surfaces
and even for smooth hypersurfaces $X$ it is very difficult to
understand the irreducible components of
$\Hom(\Pone, X)$. See \cite{h1, h2, h3, js1} for several examples.

\end{example}

Thus, in principle, we reduced the question of finding
rational curves defined over $K$ to finding $K$-points
of the scheme $\Hom(\Pone, X)$. The problem  is
that $\Hom(\Pone, X)$ is usually much more complicated than $X$.

\begin{say}[Plan to find rational curves]\label{plan.say}
We try to find rational curves defined over a field $K$ in 2 steps.
\begin{enumerate}
\item For any field $K$,  we will be able to write down reducible curves
$C$ and morphisms $f:C\to X$ defined over $K$ and show that
$f:C\to X$ can be naturally viewed as a smooth point $[f]$  in a suitable
compactification
of $\Hom(\Pone, X)$ or  $\Hom(\Pone, X, 0\mapsto x)$.

\item Then we argue that for certain fields $K$, a smooth
$K$-point in a compactification of a variety $U$ leads to a
$K$-point inside $U$.

There are 2 main cases where this works.
\begin{enumerate}
\item (Fields with an analytic inverse function theorem)

These include $\RR, \QQ_p$ or the quotient field of any
local, complete Dedekind domain, see \cite{gra-rem}. 
For such fields, any smooth  point in $\bar U(K)$
has an analytic neighborhood biholomorphic to $0\in K^n$.
This neighborhood has nontrivial intersection with any nonempty
Zariski open set, hence with $U$.

\item (Sufficiently large finite fields)

This method relies on the Lang-Weil estimates. Roughly speaking
these say that a variety $U$ over $\FF_q$ has points if $q\gg 1$,
where the bound on $q$  depends on $U$. We want to apply this to
$U=\Hom_d(\Pone, X)$. We know bounds on its embedding dimension and on
the degrees of the defining equations, but very little else.
Thus we need a form of the Lang-Weil estimates where
the bound for $q$ depends only on these invariants.
\end{enumerate}
\end{enumerate}
\end{say}

We put more detail on these steps in the next sections,
but first let us see an example.

\begin{example} Let us see what we get in a first computation
trying to find a degree 3 rational curve through a point $p$ 
on a cubic surface $S$ over $\FF_q$.

The intersection of $S$ with the tangent plane at $p$ usually gives
a rational curve $C_p$ which is singular at $p$. If we normalize
to get $n:\Pone\to C$, about half the time,  $n^{-1}(p)$ is a conjugate pair
of points in $\FF_{q^2}$. This is not what we want.

So we have to look for planes $H\subset \PP^3$ that pass through
$p$ and are tangent to $S$ at some other point.
How to count these?

Projecting $S$ from $p$ maps to $\PP^2$ and the branch  curve
$B\subset \PP^2$ has degree $4$. Moreover, $B$ is smooth
if $p$ is not on any line. The planes we are looking for 
correspond to the tangent lines of $B$.

By the Weil estimates, a degree 4 smooth plane curve has at least
$q+1-6\sqrt{q}$ points. For $q>33$ this guarantees a
point in $B(\FF_q)$ and so we get a plane $H$ through $p$
which is tangent to $S$ at some point. 

However, this is not always enough. First, we do not want the
tangency to be at $p$. Second, for any line $L\subset S$,
the plane spanned by $p$ and $L$ intersects $S$ in $L$ and a residual conic.
These correspond to double tangents of $B$. The 28 double tangents
correspond to 56 points on $B$.  Thus we can guarantee
an irreducible degree 3 rational curve
  only if we find an $\FF_q$-point on $B$ which different from these
56 points. This needs  $q>121$ for the  Weil estimates
to work. This is getting quite large!

Of course, a line is a problem only if it is
defined over $\FF_q$ and then the corresponding
 residual conic is a rational curve over $\FF_q$
passing through $p$, unless the residual conic is a 
pair of lines. In fact,  if we look for rational curves
of degree $\leq 3$, then $q>33$ works.

I do not know what the best bound for $q$ is.
In any case, we see that even this simple case leads
either to large bounds or to case analysis.

The current methods work reasonably well when $q\gg 1$,
but, even for cubic hypersurfaces,
 the bounds are usually so huge that I  do not even write them down.

Then we see by another method that for cubics we can 
handle small values of $q$. The price we pay is that
the degrees of the rational curves found end up very large.

It would be nice to figure out a reasonably sharp answer
at least for cubics,
Just to start the problem, let me say that
I do not know the answer to the following.
\end{example}

\begin{question} Let $X$ be a smooth cubic hypersurface
over $\FF_q$ and $p,p'\in X(\FF_q)$ two points.
Is there a degree $\leq 9$  rational curve defined over $\FF_q$
passing through $p$ and $p'$?
\end{question}

\begin{exercise} Let $S\subset \PP^3$ be the smooth cubic surface
constructed in (\ref{prop:one_point}).
Show that $S$ does not contain any rational curve of degree $\leq 8$ 
defined over $\FF_2$.

Hints. First prove that the Picard group of $S$ is
generated by the hyperplane sections. Thus any curve
on $S$ has degree divisible by 3.

A degree 3 rational curve would be a plane cubic, these all have
at least 2 points over $\FF_2$.

Next show that any rational curve defined over $\FF_2$
must have multiplicity 3 or more at the unique
$p\in S(\FF_2)$.  A degree 6 rational curve would be
a complete intersection of $S$ with a quadric $Q$.
Show that $S$ and $Q$ have a common tangent plane at $p$
and then prove that $S\cap Q$ has only a double point
if $Q$ is irreducible.
\end{exercise}

\section{Deformation of combs}\label{combs.sec}

\begin{example}\label{rc.over.R.exmp}
  Let  $S$ be a smooth cubic surface over $\RR$ and $p$
  a real point of $S$. Our aim is to find
a rational curve defined over $\RR$ passing through $p$.
 It is easy to find such a rational curve $C$
  defined over $\CC$. Its conjugate $\bar C$  then also passes
  through $p$.  Together, they define a curve 
$C+\bar C\subset S$ which is defined over $\RR$.
So far this is not very useful.

We can view $C+\bar C$ as the image of a map
$\phi_0: Q_0 \to S$ where $Q_0\subset \PP^2$ is  defined by $x^2+y^2=0$.
 Next we would like 
  to construct a perturbation $\phi_\varepsilon : Q_\varepsilon \to S$
  of this curve and of this map. It is easy to perturb $Q_0$ to get
``honest'' rational curves over $\RR$, for instance 
$Q_\varepsilon :=(x^2+y^2=\varepsilon z^2)$. 

The key question is, can we extend $\phi_0$ to $\phi_\varepsilon$?
Such questions are handled by     deformation theory,
originated by Kodaira and Spencer. A complete treatment of the case we
need is in \cite{rc-book} and \cite{ar-ko} is a good introduction.

The final answer is that if $H^1(Q_0, \phi_0^*T_S)=0$, then
  $\phi_\varepsilon$ exists  for $|\varepsilon|\ll 1$.
This 
  allows us to obtain a real rational curve on $S$, and
with a little care we can arrange for it to  pass
  through $p$.

In general, the above method gives the following result:
\end{example}

\begin{corollary} \cite{k-loc}
  Given $X_\RR$ such that $X_\CC$ is rationally connected, there is a real
  rational curve through any $p \in X(\RR)$.
\end{corollary}

We would like to apply a similar strategy to $X_K$ such that $X_\Kbar$ is
separably rationally connected. For a given $x \in X(K)$, we find a curve
$g_1:\Pone\to X$ defined over $\Kbar$ such that  $g_1(0)=x$, 
with conjugates $g_2, \dots,
g_m$. Then $C:=g_1(\Pone) + \dots + g_m(\Pone)$ is defined over $K$. 
Because of the
singularity of $C$ in $x$, it is harder to find a smooth deformation of $C$.
It turns out that there is a very simple way to overcome this problem:
we need to add a whole new $\Pone$ at the point $x$
and look at maps of curves to $X$ which may not be finite.

\begin{definition} Let $X$ be a variety over a field $K$.
An {\it $m$-pointed stable  curve of  genus 0 over $X$}
 is an object $(C, p_1, \dots, p_m,f)$
where 
\begin{enumerate}
\item $C$ is a proper connected   curve with $p_a(C)=0$ defined over $K$
 having only
nodes,  
\item $p_1, \dots, p_m$ are distinct smooth points in $C(K)$, 
\item  $f:C\to X$ is a $K$-morphism, and 
\item $C$ has only finitely many automorphisms that commute with $f$
and fix $p_1, \dots, p_m$.
Equivalently, there is no  irreducible component 
 $C_i\subset C_{\bar K}$ such that $f$ maps $C_i$ to a point
and $C_i$ 
 contains at most 2 special points (that is, nodes of $C$ 
or $p_1, \dots, p_m$).
\end{enumerate}

Note that if $f:C\to X$ is finite, then 
$(C, p_1, \dots, p_m,f)$ is a stable  curve of  genus 0 over $X$,
even if $(C, p_1, \dots, p_m)$ is not a stable
$m$-pointed  genus 0 curve in the usual sense \cite{ful-pan}.

We have shown how to parametrize all maps $\Pone\to X$
by the points of a scheme $\Hom(\Pone, X)$.
Similarly, the methods of  \cite{konts} and \cite{alexeev} 
show  that one can 
parametrize all $m$-pointed genus 0 stable curves of degree $d$ with a single
scheme $\overline{M}_{0,m}(X,d)$. 
For a map $(C, p_1, \dots, p_m,f)$,
the corresponding point in $\overline{M}_{0,m}(X,d)$
is denoted by $[C, p_1, \dots, p_m,f]$.

Given $K$-points $x_1,\dots, x_m\in X(K)$,
the family of those maps $f:C\to X$ that satisfy
$f(p_i)=x_i$ for $i=1,\dots, m$ forms a closed
 scheme 
$$
\overline{M}_{0,m}(X,p_i\mapsto x_i)\subset \overline{M}_{0,m}(X).
$$

See \cite[sec.8]{ar-ko} for more detailed proofs.
\end{definition}

The deformation theory that we need can be conveniently
compacted into one statement. The result basically says
that the deformations used in (\ref{rc.over.R.exmp})
exist for any reducible rational curve.

\begin{theorem} \label{m0mbar.thm}
(cf.\ \cite[Sec.II.7]{rc-book} or \cite{ar-ko}) Let 
$f:(p_1,\dots, p_m\in C)\to X$ be  an $m$-pointed genus 0 stable curve.
Assume that $X$ is smooth and   
$$
H^1(C, f^*T_X(-p_1-\cdots-p_m))=0.
$$
Then: 
\begin{enumerate}
\item There is a unique irreducible component
$$
{\rm Comp}(C,p_1,\dots,p_m,f)\subset \overline{M}_{0,m}(X,p_i\mapsto f(p_i))
$$
which contains $[C,p_1,\dots,p_m,f]$.
\item $[C,p_1,\dots,p_m,f]$ is a smooth point of
 ${\rm Comp}(C,p_1,\dots,p_m,f)$.
In particular, if $f:(p_1,\dots,p_m\in C)\to X$
is defined over $K$ then ${\rm Comp}(C,p_1,\dots,p_m,f)$
  is geometrically irreducible.
\item There is a dense open subset
$$
\smooth(C,p_1,\dots,p_m, f)\subset {\rm Comp}(C,p_1,\dots,p_m,f)
$$
which parametrizes free maps of smooth rational curves, that is
$$
\smooth( C, p_1,\dots,p_m, f)
\subset\Hom^{\rm free}(\PP^1, X, p_i\mapsto f(p_i))
$$
\end{enumerate}
\end{theorem}

(We cheat a little in (\ref{m0mbar.thm}.2). In general
$[C,p_1,\dots,p_m,f]$ is smooth only in the stack sense;
this is all one needs.
Moreover,
in all our applications $[C,p,f]$ will be a smooth point.)

The required vanishing is usually easy to check using the following.

\begin{exercise}\label{vanish.lem} Let $C=C_1+\cdots+C_m$
 be a reduced, proper  curve with arithmetic genus $0$
and $p\in C$ a smooth point.
Let $C_1,\cdots ,C_m$ be its irreducible components over $\bar K$.
Let $E$ be a vector bundle on $C$
and assume that $H^1(C_i, E|_{C_i}(-1))=0$ for every $i$.
Then  $H^1(C,E(-p))=0$.

In particular,  if $f:C\to X$ is a morphism to a smooth variety
and if each $f|_{C_i}$ is free then  
 $H^1(C, f^*T_X(-p))=0$.
\end{exercise}

\begin{definition}[Combs]\label{comb.say}

A  {\it comb} assembled from a curve 
$B$ (the handle) and $m$ curves $C_i$ (the teeth) 
 attached
 at the distinct points $b_1,\dots, b_m\in B$
and $c_i\in C_i$   is a curve    obtained from
the disjoint union of $B$ and of the $C_i$
by identifying the points $b_i\in B$ and $c_i\in C_i$.
In these notes we only deal with the case when $B$ and the $C_i$
 are smooth, rational.

A comb can be pictured as below:
$$
\begin{array}{c}
\begin{picture}(100,100)(40,-70)
\put(0,0){\line(1,0){180}}
\put(20,10){\line(0,-1){60}}
\put(50,10){\line(0,-1){60}}
\put(160,10){\line(0,-1){60}}
\put(10,5){$b_1$}
\put(40,5){$b_2$}
\put(150,5){$b_m$}

\put(20,0){\circle*{3}}
\put(50,0){\circle*{3}}
\put(160,0){\circle*{3}}

\put(15,-60){$C_1$}
\put(45,-60){$C_2$}
\put(155,-60){$C_m$}

\put(-25,-5){$B$}

\put(90,-30){$\cdots\cdots$}

\end{picture}\\
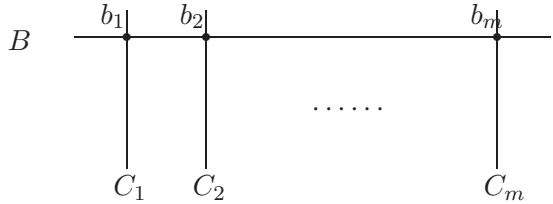
\mbox{Comb with $m$-teeth}
\end{array}
$$
Assume now that we have 
a Galois extension $L/K$ and
$g_i:(0\in \Pone)\to (x\in X)$, a conjugation invariant set of maps
 defined over $L$.

We can view this collection as just one map as follows.
The maps $[g_i]\in \Hom(\Pone, X)$ form a
0-dimensional reduced $K$-scheme $Z$. Then the
$g_i$ glue together to a single map
$$
G: Z\times (0\in \Pone)=(Z\subset \Pone_Z)\to (x\in X).
$$

Let $j:Z\into \Pone_K$ be an embedding. 
We can then assemble a comb with handle $\Pone_K$
and teeth $\Pone_Z$. Let us denote it by
$$
\comb(g_1,\dots , g_m).
$$
(The role of $j$ is suppressed, it will not be important for us.)

If $K$ is infinite, an embedding  $j:Z\into \Pone_K$
always exists. If $K$ is finite, then $Z$ may have too many
points, but an embedding exists whenever $Z$ is irreducible over $K$.

Indeed, in this case $Z=\spec_KK(a)$ for some $a\in \bar K$.
Thus $K[t]\to K(a)$ gives an embedding $Z\into \AA^1_K$.
\end{definition}

Everything is now ready to obtain
rational curves through 1 point.

\begin{corollary}\cite{k-loc}\label{k-loc.cor}
  Given a separably rationally connected variety $X$ defined over 
a local field $K=\QQ_p$ or $K=\FF_q((t))$,
  there is a rational curve defined over $K$ through any $x \in X(K)$.
\end{corollary}

Proof. Given $x \in X(K)$,  pick a free curve
$g_1:(0\in \Pone)\to (x\in X)$ over $\bar K$ with conjugates $g_2, \dots,
g_m$. As in (\ref{comb.say}), assemble a $K$-comb
$$
f:(0\in \comb(g_1,\dots , g_m))\to (x\in X).
$$
Using (\ref{m0mbar.thm}), 
we obtain
$$
\smooth(C,0, f)\subset\Hom^{\rm free}(\PP^1, X, 0\mapsto f(p))
$$
and by (\ref{m0mbar.thm}.2)
we see that  (\ref{plan.say}.2.a)
applies. Hence  we get $K$-points in $\smooth(p\in C, f)$, as required.
\qed
\medskip

The finite field case, corresponding to
(\ref{plan.say}.2.b), is treated in the next section.

\section{The Lang-Weil estimates} \label{lw.sec}

\begin{theorem}\cite{lan-wei}\label{lan-wei.thm}
  Over $\FF_q$, let $U \subset \PP^N$ be the difference of two subvarieties
  defined by several equations of degree at most $D$. 
If $U_0\subset U$ is a geometrically
  irreducible component, then
  \[\left|\# U_0(\FF_q)-q^{\dim U_0}\right|
  \le C(N,D) \cdot q^{\dim U_0 - \frac 1 2},\]
where the constant $C(N,D)$ depends only on $N$ and $D$.
\end{theorem}

{\it Notes on the proof.} The original form of the estimate in \cite{lan-wei}
assumes that $U_0$ is projective and it uses
$\deg U_0$ instead of $D$. These are, however, minor changes.

First, if $V\subset \PP^N$ is an  irreducible component
of $W$ which is defined by equations of degree at most $D$,
then it is also an  irreducible component
of some $W'\supset W$ which is defined by 
$N-\dim V$ equations of degree at most $D$.
Thus, by B\'ezout's theorem,  $\deg V\leq D^{N-\dim V}$.

Thus we have a bound  required for  $\# \bar U_0(\FF_q)$
and we need an upper bound for the complement
$\# (\bar U_0\setminus U_0)(\FF_q)$. We assumed that
$\bar U_0\setminus U_0$ is also defined by equations of degree at most $D$.
A slight problem is, however, that it may have components which are
geometrically reducible.  Fortunately, an upper bound for
$\# V(\FF_q)$ is easy to get.\qed

\begin{exercise} Let $V\subset \PP^N$ be a closed, reduced subscheme of
pure dimension $r$ and degree $d$.
Show that if $q\geq d$ then $V$ does not contain
$\PP^N(\FF_q)$. Use this to show that there is a projection
$\pi:V\to \PP^r$ defined over $\FF_q$
which is finite of degree $d$.
Conclude from this that
$$
\# V(\FF_q)\leq d\cdot \# \PP^r(\FF_q)=d\cdot (q^r+\cdots+q+1).
$$
\end{exercise}

\begin{say}[Application to $\Hom_d(\Pone, X)$]\label{lw.appl.say}
We are looking for rational curves of degree $d$
on a hypersurface $X\subset \PP^{n+1}$ of degree $m$.
We saw in (\ref{sp.of.rc.say}) that 
 $\Hom_d(\Pone, X)$ lies in $\PP^{(n+2)(d+1)-1}$
(hence we can take $N=(n+2)(d+1)-1$) and its closure is defined by
equations of degree $m$. 

The complement of $\Hom_d(\Pone, X)$ in its closure
consists of those 
$(f_0, \dots, f_N)$ with a common zero.
One can get explicit equations for this locus
 as follows.
Pick indeterminates $\lambda_i, \mu_j$. Then
$f_0, \dots, f_N$ have a common zero iff the resultant
$$
\mbox{Res}(\textstyle{\sum_i} \lambda_i f_i,\textstyle{\sum_j} \mu_j f_j)
$$
is identically zero as a polynomial in the $\lambda_i, \mu_j$.
This gives equations of degree $2d$ in the coefficients of the $f_i$.
Thus we can choose $D=\max\{m,2d\}$.

Finally, 
where do we find a geometrically
  irreducible component of the space  $\Hom_d(\Pone, X)$?
Here again a smooth point $[f]$  in a suitable
compactification
of $\Hom(\Pone, X)$ gives the answer by (\ref{smooth.pt.exrc}).
Similar considerations show that our methods
also apply to  $\Hom_d(\Pone, X, 0\mapsto p)$.
\end{say}

\begin{exercise} \label{smooth.pt.exrc}
 Let $W$ be a $K$-variety and $p\in W$ a smooth point.
Then there is a unique $K$-irreducible component $W_p\subset K$
which contains $p$ and $W_p$ is also geometrically
  irreducible if either $p$ is $K$-point
or $K$ is algebraically closed in $K(p)$.
\end{exercise}

As a first application, let us consider cubic surfaces.

\begin{example}[Cubic surfaces]\label{cubic.surf.1pt.exmp}

Consider a cubic surface $S \subset \Pthree$, defined over $K=\FF_q$. 
We would like to use these results to get a rational curve through 
 any $p\in S(\FF_q)$.

We need to start with some free rational curves over
$\bar K$.

The first such possibility is to use conics.
If $L\subset S$ is a line, then
the plane spanned by $p$ and $L$ intersects $S$ in $L$ plus a
residual conic $C_L$. $C$ is a smooth and free conic,
unless $p$ lies on a line. 

In general, we get 27 conics and we conclude that
if $q$ is large enough, then through every point $p \in S(\FF_q)$
which is not on a line, there is
  rational curve of degree  $2\cdot 27=54$, defined over $\FF_q$.

If $p$ lies on 1 (resp.\  2) lines, then we get only
16 (resp.\ 7) smooth conics, and so we get even lower
degree rational curves.

However, when  $p$ lies on 3 lines (these are called Eckart points)
then there is no smooth conic through $p$.

Let us next try twisted cubics.
As we saw in (\ref{ex:cubic_curves}), we get
twisted cubics from a morphism $S\to \Ptwo$
as the birational transforms of lines not passing through any
of the 6 blown up points.
Thus we get a 2-dimensional family of twisted cubics
whenever $p$ is not on one the 6 lines
contracted by $S\to \Ptwo$.

If $p$ lies on 0 (resp.\ 1, 2, 3) lines,
we get $72$ (resp.\ $72-16$,  $72-2\cdot 16$, 
 $72-3\cdot 16$) such families. 

Hence we obtain that for every $p\in S(\FF_q)$, the space
$\Hom_d(\Pone, X, 0\mapsto p)$ has a geometrically irreducible
component for some $d\leq 3\cdot 72=216$.

As in (\ref{lw.appl.say}), we conclude that
 if $q$ is large enough, then through every point 
$p \in S(\FF_q)$, there is
  rational curve of degree at most $216$, defined over $\FF_q$.
\end{example}

\begin{example}[Cubic hypersurfaces]

Consider a smooth cubic hypersurface 
$X^n \subset \PP^{n+1}$, defined over $K=\FF_q$
and let $p\in X(\FF_q)$ be a point.

If $p$ lies on a smooth cubic surface section $S\subset X$,
then we can assemble a $K$-comb of degree $\leq 216$
and, as before,  we can use it to get rational curves
through $p$.

Over a finite field, however, there is no guarantee that
$X$ has any smooth cubic surface sections. What can we do then?

We can use a generic cubic surface section through $p$.
This is then defined over a field extension
$L=K(y_1, \dots, y_s)$ where the $y_i$ are algebraically
independent over $K$. By the previous considerations we
can assemble an $L$-comb and conclude that 
$\Hom_d(\Pone, X, 0\mapsto p)$ has a  smooth $L$-point
for some $d\leq 3\cdot 72=216$.

By (\ref{smooth.pt.exrc}), this implies that it also has a 
geometrically irreducible
component, and we can then finish as before.
\end{example}

It is now clear that the methods of this section together with
(\ref{low.deg.rtl.curves}) imply the following:

\begin{theorem}\label{one.pt.main.thm}
 Let  $X\subset \PP^{n+1}$ be a smooth SRC
hypersurface of degree $m\leq n+1$ defined over a finite field $\FF_q$.
Then there is a $C(n)$ such that if $q> C(n)$ then
through every point in $X(\FF_q)$ there is a rational curve
defined over $\FF_q$.\qed
\end{theorem}

\begin{exercise} \label{becomes.onto.exrc}
Prove the following consequence of (\ref{lan-wei.thm}):

Let $f:U\to W$ be a dominant morphism over $\FF_q$. Assume that  $W$
and the generic fiber of $f$ are both 
geometrically irreducible. Then there is a dense open set
$W^0$ such that 
$f\bigl(U(\FF_{q^m})\bigr)\supset W^0(\FF_{q^m})$ for $m\gg 1$.
\end{exercise}

\section{Rational curves through two  points and\newline
Lefschetz-type theorems}\label{2pt.sec}

\begin{say}[How not to find rational curves through two points]

Let us see what happens if we try to follow the method of (\ref{k-loc.cor})
for 2 points.
Assume that  over $\Kbar$  we
have a rational curve $C_1$ through $p, p'$. Then $C_1$ 
 is already defined
 over a finite Galois extension $K'$ of $K$. As
before, consider its conjugates of $C_2, \dots, C_m$  under
$G:=\Gal(K'/K)$, and attach copies $C'_1, \dots, C'_m$  to \emph{two}
copies of $\Pone$, one over $p$ and one over $p'$.
 This results in a curve $Y_0$ which is  defined over
$K$ and may be deformed to a smooth curve $Y_\varepsilon$,
still passing  through $p,p'$.

The problem is that although all the $\Kbar$-irreducible components of
$Y_0$ are rational, it has arithmetic genus $m-1$, hence the
smooth curve $Y_\varepsilon$ has genus $m-1$.

  Note that finding curves of higher genus through $p, p'$ is not very
  interesting. Such a curve can easily be obtained by taking the intersection
  of $X$ with hyperplanes through $p, p'$.

In fact, no other choice of $Y_0$ would work, as shown by the next exercise.
\end{say}

\begin{exercise}  Let $C$ be a reduced, proper, connected  curve
of arithmetic genus $0$ 
defined over $K$. 
Let $p\neq  p'\in C(K)$ be 2 points.
Then there is a closed sub-curve $p,p'\in C'\subset C$ such that
$C'$ is connected and every $K$-irreducible component of $C'$
is isomorphic to $\PP^1_K$.
\end{exercise}

In this section we first connect the existence of
rational curves through two points with Lefschetz-type results
about the fundamental groups of open subsets of $X$ and then use this
connection to find such rational curves in certain cases.

\begin{definition}\label{leff.cond.defn}
Let $K$ be a field,
 $X$  a normal, projective variety
and
$$
\begin{array}{ccc}
C_U & \stackrel{\phi}{\to} & X\\
\pi\downarrow\uparrow s && \\
U &&
\end{array}\eqno{(\ref{leff.cond.defn}.1)}
$$
a  smooth
family of reduced, proper, connected  curves mapping to  $X$
with a section $s$. 
For $x\in X$, set $U_{s\to x}:=s^{-1}\phi^{-1}(x)$, parametrizing
those maps that send the marked point to $x$,
and
$$
\begin{array}{ccc}
C_{U_{s\to x}} & \stackrel{\phi_x}{\to} & X\\
\pi_x\downarrow\uparrow s_x && \\
U_{s\to x} &&
\end{array}\eqno{(\ref{leff.cond.defn}.2)}
$$
the corresponding family.

We say that  the family (\ref{leff.cond.defn}.1) satisfies the 
 {\it Lefschetz condition}
if, for general $x\in X(\bar K)$, 
the map  $\phi_x$ is dominant with 
geometrically irreducible generic fiber.

Sometimes it is more convenient to give just
$$
U\stackrel{\pi}{\leftarrow} C_U\stackrel{\phi}{\to} S,
\eqno{(\ref{leff.cond.defn}.3)}
$$
without specifying the section $s:U\to C_U$. In this case,
we consider the family obtained from the universal section. That is,
$$
\begin{array}{ccc}
C_U\times_U C_U & \stackrel{\phi_1}{\to} & X\\
\pi_2\downarrow\uparrow s_1 && \\
C_U &&
\end{array}\eqno{(\ref{leff.cond.defn}.4)}
$$
where $\pi_2(c,c')=c', \phi_1(c,c')=\phi(c)$ and
$s_1(c)=(c,c)$.

If $x\in X$ then  $U_x=\phi^{-1}(x)=s_1^{-1}\phi_1^{-1}(x)$
is the set of triples $(C,c, \phi|_C)$ where
$C$ is a fiber of $\pi$ and $c$ a point of $C$ such that
$\phi(c)=x$.

Similarly, if $x,x'\in X$ then  
$U_{x,x'}:=\phi^{-1}(x)\times_U \phi^{-1}(x') $
is the set of all $(C,c,c', \phi|_C)$ where
$C$ is a fiber of $\pi$ and $c,c'$  points of $C$ such that
$\phi(c)=x$ and $\phi(c')=x'$. Informally
(and somewhat imprecisely) 
$U_{x,x'}$ is the family of curves in $U$ that pass through
both $x$ and $x'$.

Thus the family (\ref{leff.cond.defn}.4) 
satisfies the  Lefschetz condition iff
$(C_U)_{x,x'}$ is  geometrically irreducible for
general $x,x'\in X(\bar K)$.
\end{definition}

\begin{exercise}[Stein factorization] \label{stein.exrc}
Let $g:U\to V$ be a morphism between irreducible and normal varieties.
Then $g$ can be factored as
$$
g: U \stackrel{c}{\to} W\stackrel{h}{\to}V
$$
where $W$ is normal, $h$ is finite and generically \'etale 
and there is an open and dense subset
$W^0$ such that 
$c^{-1}(w)$ is geometrically irreducible for every $w\in W^0$.

Thus $g:U\to V$ is dominant with 
geometrically irreducible generic fiber iff
$g$ can not be factored through a
nontrivial finite and generically \'etale  map
$W\to V$.
\end{exercise}

The Appendix explains how the Lefschetz condition
connects with the Lefschetz theorems on fundamental groups
of hyperplane sections. For now let us   prove that
 a family satisfying  the  Lefschetz condition
leads to rational curves through 2 points.

\begin{example}\label{leff.for.3surf}
 Let $S\subset \PP^3$ be a smooth cubic surface.
Let $U\leftarrow C_U\to S$ be
 the family of rational hyperplane sections.
Note that $C_U\to S$ is dominant with 
geometrically irreducible generic fiber.
Furthermore, for general $p\in S(\bar K)$, 
the map  $\phi_p$ is dominant, generically finite and has degree 12.

On the other hand, let $U$ be an irreducible family of
twisted cubics on $S$. Then $U$ satisfies the  Lefschetz condition.
As discussed in (\ref{ex:cubic_curves}), 
 $U$ corresponds to the family
of lines in $\PP^2$ not passing through the 6 blown-up points.
Thus $U_x$ consists of lines in $\PP^2$ through $x$,
hence $\phi_x: C_{U_x}\to S$ is birational. Thus it cannot factor
through a nontrivial finite cover.
\end{example}

\begin{theorem}\label{2pt.curve.leff.thm} 
Let $X$ be a smooth projective variety over $\FF_q$.
Let $U\subset \Hom^{\rm free}(\Pone, X)$ be a geometrically
irreducible smooth subset, closed under ${\rm Aut}(\Pone)$.
Assume that 
$U\stackrel{\pi}{\leftarrow} U\times \Pone\stackrel{\phi}{\to} X$
satisfies the   Lefschetz condition.

Then there is an open subset $Y^0\subset X\times X$ such that
for $m\gg 1$ and $(x,x')\in Y^0(\FF_{q^m})$
there is a point $u\in U(\FF_{q^m})$ giving a rational curve
$$
\phi_u:\Pone\to X \quad\mbox{such that}\quad \phi_u(0)=x,\ \phi_u(\infty)=x'.
$$
\end{theorem}

Proof. Set $s(u)=(u,0)$ and consider the map
$$
\Phi_2:=(\phi\circ s\circ \pi, \phi):U\times \Pone \to X\times X.
$$
Note that on $U_{s\to x}\times \Pone$ this is just $\phi_x$ followed by
the injection $X\cong \{x\}\times X\into X\times X$.

If the   generic fiber of $\Phi_2$ is
geometrically irreducible, then by
 (\ref{stein.exrc})  and (\ref{becomes.onto.exrc}),
there is an open subset $Y^0\subset X\times X$ such that
for $m\gg 1$ and
for every $(x,x')\in Y^0(\FF_{q^m})$ there is a
$(u,p)\in U(\FF_{q^m})\times \Pone(\FF_{q^m})$
such that $\Phi_2(u,p)=(x,x')$. 
This means that  $\phi_u(0)=x$ and $\phi_u(p)=x'$. 
A suitable automorphism $\gamma$ of $\Pone$ sends
$(0,\infty)$ to $(0,p)$. Thus $\phi_u\circ \gamma$
is the required rational curve.

If the generic fiber of $\Phi_2$ is
geometrically reducible, then $\Phi_2$ factors through
a nontrivial finite cover $W\to X\times X$. 
For general $x\in X$, the restriction ${\rm red} W_x\to \{x\}\times X$
is nontrivial and $U_{s\to x}\to {\rm red} W_x$ is dominant.
This is impossible by the  Lefschetz condition.\qed
\medskip

Next we discuss how to construct families that satisfy the
Lefschetz condition.

\begin{lemma} \label{leff.smoothing.ok}
Let $U\stackrel{\pi}{\leftarrow} C_U\stackrel{\phi}{\to} X$
be a smooth family of reduced, proper, generically irreducible 
 curves over $\bar K$
such that $U_x$ is irreducible for general $x\in X$.
 Let
$W\subset U$ be a locally closed smooth subset and 
$W\times \Pone \cong D_W\subset C_W$  a
subfamily. 
Let $U^0\subset U$ be an open dense  subset.
If 
$$
W\stackrel{\pi}{\leftarrow} D_W\stackrel{\phi}{\to} X
$$
satisfies the   Lefschetz condition, then so does
$$
U^0\stackrel{\pi}{\leftarrow} C_{U^0}\stackrel{\phi}{\to} X.
$$
\end{lemma}

Proof.  Assume that contrary. Then there is a
nontrivial finite and generically \'etale  map
$Z\to X$ such that the restriction $\phi|_{C_{U^0_x}}: C_{U^0_x}\to X$ 
factors through $Z$. Since $U_x$ is irreducible, so is $Z$.

 Let
$g_C:C_U\times_XZ\to C_U$ be the projection. By assumption,
there is a rational section  $s:C_{U^0_x}\to C_U\times_XZ$.
Let $B\subset C_U\times_XZ$ be the  closure of its image.
Then $g_C|_B:B\to C_U$ is finite and an isomorphism over $C_{U^0_x}$.
Thus  $g_C|_B:B\to C_U$ is an isomorphism at every point
where $C_U$ is smooth (or normal). 
In particular, $s$ restricts to a rational section
$s_W:D_W\rto C_U\times_XZ$. 

Repeating the previous argument, we see that $s_W$ is
an everywhere defined section, hence
$\phi|_{D_W}$ factors through $Z$, a contradiction.\qed

\begin{corollary} Let $X$  be 
a smooth projective variety over a perfect field $K$.
If there is a $\bar K$-family of free curves
$$
U_1\stackrel{\pi_1}{\leftarrow} U_1\times \Pone\stackrel{\phi_1}{\to} X
$$
satisfying the  Lefschetz condition then there is a
$K$-family of free curves
$$
U\stackrel{\pi}{\leftarrow} U\times \Pone\stackrel{\phi}{\to} X
$$
satisfying the  Lefschetz condition.
\end{corollary}

Proof.  As usual, the first family is defined over a finite Galois extension;
let $U_1,\dots, U_m$ be its conjugates.

Consider the family of all $\bar K$-combs
$$
\comb(U):=\bigl\{\comb(\phi_{1,u_1}, \dots, \phi_{m,u_m})\bigr\}
$$
where $u_i\in U_i$ and 
$\phi_1(u_1,0)=\cdots=\phi_m(u_m,0)$
with $0$ a marked point on the handle. 
(We do not assume that the $u_i$ are
conjugates of each other.)
Each comb is defined by choosing $u_1,\dots, u_m$ as above
and $m$ distinct points in $\Pone\setminus\{0\}$. 

Thus $\comb(U)\subset \overline{M}_{0,1}(X)$ is defined over $K$.
Furthermore, for each $x\in X$, 
$\comb(U)_x\subset \overline{M}_{0,1}(X, 0\mapsto x)$ is
isomorphic to an open subset of 
$$
(\Pone)^m\times 
U_{1,x}\times \cdots \times U_{m,x},
$$
hence irreducible.

By (\ref{m0mbar.thm}), there is a unique irreducible component 
${\rm Smoothing}(U)\subset  \overline{M}_{0,1}(X)$ 
containing $\comb(U)$ and ${\rm Smoothing}(U)$ is defined
over $K$.

We can now apply (\ref{leff.smoothing.ok})
with $W:=\comb(U)$ and $D_W\to W$ the first tooth
of the corresponding comb. This shows that
${\rm Smoothing}(U)$ satisfies the Lefschetz condition.
\qed

\begin{example}[Cubic hypersurfaces]  
We have already seen in (\ref{leff.for.3surf})
how to get a family of rational curves on a  smooth cubic surface $S$ that
satisfies the Lefschetz condition:

For general $p\in S$, there are 72 one-parameter families of
twisted cubics
$C_1,\dots, C_{72}$ through $p$. Assemble these into a 1-pointed comb
and smooth them to get a family $U(S)$ of degree $216$ rational curves.
(In fact, the  
family  of degree $216$ rational curves on $S$ that are
linearly equivalent to $\OO_S(72)$ is irreducible,
and so equals $U(S)$,  but we do not need this.)

Let us go now to a higher dimensional cubic $X\subset \PP^{n+1}$.
Let $G$ denote the Grassmannian of $3$-dimensional
linear subspaces in $\PP^{n+1}$. Over $G$ we have 
${\mathbf S}\to G$, the
universal family of cubic surface sections of $X$.
For any fiber $S=L^3\cap X$ we can take
$U(S)$. These form a family of rational curves ${\mathbf U(\mathbf S)}$ on $X$
and we obtain
$$
{\mathbf U(\mathbf S)}\stackrel{\pi}{\leftarrow} 
{\mathbf U(\mathbf S)}\times \Pone 
\stackrel{\phi}{\to} X.
$$
We claim that it satisfies the Lefschetz condition.
Indeed, given $x,x'\in X$, the family of
curves  in ${\mathbf U(\mathbf S)}$ that pass through $x,x'$
equals
$$
{\mathbf U(\mathbf S)}_{x,x'}=
\bigcup_{x,x'\in L^3} \bigl(U(L^3\cap X)\bigr)_{x,x'}.
$$
The set of all such $L^3$-s is parametrized by
the Grassmannian of 
lines in $\PP^{n-1}$, hence geometrically irreducible.
The general  $L^3\cap X$ is a smooth cubic surface, hence
we already know that the corresponding
$U(L^3\cap X)_{x,x'}$ is irreducible.
 Thus
${\mathbf U(\mathbf S)}_{x,x'}$ is irreducible.
\end{example}

Although we did not use it for cubics, let us note
the following.

\begin{theorem}\cite{rcfg1, rcfg2}
 Let $X$ be a smooth, projective SRC variety over a  field $K$.
Then there is a family of rational curves
defined over $K$
$$
U\stackrel{\pi}{\leftarrow} U\times \Pone\stackrel{\phi}{\to} X
$$
that satisfies the  Lefschetz condition.
\end{theorem}

\begin{say}[Going from 2 points to many points]\label{2-to-mpts.say}

It turns out that going from curves passing through
2 general points to curves passing through $m$ arbitrary
points does not require  new ideas.

Let us see first how to find a curve through 2 arbitrary points
$x, x'\in X$. 

We have seen in Section \ref{rcpar.sec} how to produce
very free curves in $\Hom(\Pone, X, 0\mapsto x)$
and in $\Hom(\Pone, X, 0\mapsto x')$. If $m\gg 1$
then we can find $\psi\in \Hom(\Pone, X, 0\mapsto x)$
and  $\psi'\in \Hom(\Pone, X, 0\mapsto x')$
such that (\ref{2pt.curve.leff.thm}) produces a
rational curve $\phi:\Pone\to X$ passing through
$\psi(\infty)$ and $\psi'(\infty)$. 

We can view this as a length 3 chain 
$$
(\psi,\phi,\psi'): \Pone\vee\Pone\vee\Pone \to X
$$
through $x,x'$.
Using (\ref{m0mbar.thm}), we get a family of free 
rational curves through $x,x'$ and, again for $m\gg 1$
a single free curve through $x,x'$.

How to go from 2 points to $m$ points $x_1,\dots, x_m$?
For each $i>1$ we already have very free curves a
$g_i:\Pone\to X$ such that $g_i(0)=x_1$ and $g_i(\infty)=x_i$.
We can assemble a comb with $(m-1)$ teeth
$\phi: \comb(g_2,\dots,g_m)\to X$.

By (\ref{m0mbar.thm}), we can smooth it in
 $$
\overline{M}_{0,m}(X, p_1\mapsto x_1, \cdots, p_m\mapsto x_m)
$$
to get such rational curves.
\end{say} 

\subsection*{Appendix. The Lefschetz condition and fundamental groups}{\ }

The classical Lefschetz theorem says that if $X$ is a smooth, projective 
variety over $\CC$ and
$j:C\into X$ is a smooth curve obtained by intersecting
$X$ with hypersurfaces, then the natural map
$$
j_*:\pi_1(C)\to \pi_1(X)\quad\mbox{is onto.}
$$
Later this was extended to $X$ quasi-projective. Here $j_*$
need not be onto for every  curve section $C$, 
 but $j_*$ is onto  for  general curve sections.
In particular we get the following. (See \cite{gor-mac}
for a general discussion and further results.)

\begin{theorem}\label{qp.leff.class} Let $X^n$ be a smooth, projective 
variety over $\CC$ and $|H|$ a very ample linear system.
Then, for every open subset $X^0\subset X$
and general $H_1,\dots, H_{n-1}\in |H|$,
$$
\pi_1\bigl(X^0\cap H_1\cap\dots\cap H_{n-1}\bigr)\to 
\pi_1(X)\quad\mbox{is onto.}
$$
It should be stressed that the notion of ``general'' depends on $X^0$.
\end{theorem}

If $X$ is a  hypersurface of degree $\geq 3$ then
 the genus of the curves $H_1\cap\dots\cap H_{n-1}$ is at least 1.
We would like to get a similar result where
$\{H_1\cap\dots\cap H_{n-1}\}$ is replaced by some family of rational curves.

 The following argument shows that 
if a family of curves satisfies the  Lefschetz condition, then
 (\ref{qp.leff.class}) also holds for that family. 

Pick a family of curves
$U\stackrel{\pi}{\leftarrow} C_U\stackrel{\phi}{\to} X$
with a section $s:U\to C_U$ that satisfies the  Lefschetz condition.

Given a generically \'etale 
$g:Z\to X$, there is an open $X^0\subset X$ such that
$Z^0:=g^{-1}(X^0)\to X^0$ is finite and \'etale.

Pick a general point $p\in X^0$. There is an open subset
$U_p^0\subset U_p$ such that $\phi_p^{-1}(X^0)\to U_p$
is topologically a locally trivial fiber bundle over $C^0_{U_p}\to U_p^0$
with typical fiber $C^0_u=C_u\cap \phi^{-1}(X^0)$
where $u\in U$ is a general point.

Thus there is a right split  exact sequence
$$
\pi_1(C^0_u)\to \pi_1(C^0_{U_p})\leftrightarrows \pi_1(U_p^0)\to 1,
$$
where the splitting is given by the section $s$.
Since $s(U_p)$  gets mapped to the point $p$ by $\phi$,
 $\pi_1(C^0_{U_p})$ gets killed in
$\pi_1(X^0)$. Hence
$$
{\rm im} [\pi_1(C^0_u)\to \pi_1(X^0)]=
{\rm im}[\pi_1(C^0_{U_p})\to \pi_1(X^0)].
$$
Since $C_{U_p}\to X$ is dominant, 
${\rm im}[\pi_1(C^0_{U_p})\to \pi_1(X^0)]$ has finite index in
$\pi_1(X^0)$. We are done if the image is $\pi_1(X^0)$.
Otherwise the image corresponds to a nontrivial
 covering $Z^0\to X^0$  and 
$\phi_p$ factors through $Z^0$.
This, however, contradicts the  Lefschetz condition.\qed
\medskip

A more detailed consideration of the above argument shows
that (\ref{qp.leff.class}) is equivalent to the following
weaker Lefschetz-type conditions:
\begin{enumerate}
\item The generic fiber of $C_U\to X$ is geometrically irreducible, and
\item for general $x\in X$, $U_x$ is geometrically irreducible  and
$C_{U_x}\to X$ is dominant.
\end{enumerate}

In positive characteristic the
above argument has a problem with the claim that
something is ``topologically a locally trivial fiber bundle''
and indeed the two versions are not quite equivalent.
In any case, the purely algebraic version of
(\ref{leff.cond.defn}) works better for us.

\section{Descending from $\FF_{q^2}$ to  $\FF_q$}\label{desc.sect}

Our methods so far constructed rational curves on 
hypersurfaces over $\FF_q$ for $q\gg 1$. Even for
cubics, the resulting bounds on $q$ are huge.
The aim of this section is to use the
third intersection point map to construct
rational curves on cubic hypersurfaces over $\FF_q$
from rational curves on cubic hypersurfaces over $\FF_{q^2}$.
The end result is a proof of (\ref{cubic.ext.thm}). 
The price we pay is that the degrees of the rational curves
become larger as $q$ gets smaller.

\begin{say}[Descent method]\label{desc.meth.say}
  Let $X$ be a cubic hypersurface, $C$ a smooth curve 
and $\phi:C(\FF_q)\to X(\FF_q)$ a map of sets.

Assume that for each $p\in C(\FF_q)$ there
is a line $L_p$ through $\phi(p)$
which intersects $X$ in two further points
$s(p), s'(p)$. These points are in $\FF_{q^2}$ and
we assume that none of them is in $\FF_q$, hence 
$s(p), s'(p)$ are conjugate over $\FF_q$. 
This gives a lifting of $\phi$
to $\phi_2:C(\FF_q)\to X(\FF_{q^2})$ where
$\phi_2(p)=s(p)$. (This involves a choice for each $p$ but this
does not matter.)

Assume that over $\FF_{q^2}$ there is an extension of $\phi_2$
to $\Phi_2:C\to X$. If $\bar\Phi_2$ denotes the conjugate map,
then $\bar \Phi_2(p)=s'(p)$.

Applying the third intersection point map 
(\ref{tipm.say})  to
the Weil restriction (\ref{ex:weil_restriction}) we get an $\FF_q$-map
$$
h:  \R_{\FF_{q^2}/\FF_q}C\to X.
$$
Since $C$ is defined over $\FF_q$, the Weil restriction
has a diagonal
$$
j:C\into \R_{\FF_{q^2}/\FF_q}C
$$
and $\Phi:=h\circ j: C\to X$  is the required lifting
of $\phi$.

Thus, in order to prove (\ref{cubic.ext.thm}), we need to show that
\begin{enumerate}
\item (\ref{cubic.ext.thm}) holds   for $q\gg 1$, and
\item for every $x\in X(\FF_q)$ there is a line $L$ as required.
\end{enumerate}
\end{say}

\begin{remark} In trying to use the above method over
an arbitrary field $K$, a significant problem is
that for  each point $p$ we get a degree 2 field extension
$K(s(p))/K$ but we can  use these only if they are
all the same. A finite field has a unique extension of any given
degree, hence the extensions $K(s(p))/K$ are automatically the same.

There are a few other fields with 
a unique degree 2 extension,
for instance  $\RR, \overline{\QQ}((t))$
or $\overline{\FF}_p((t))$ for $p\neq 2$.

If we have only 1 point $p$, then the method works
over any field $K$. This is another illustration
that the 1 point case is much easier.

In the finite field case, the method can also deal with
odd degree points of $C$ but not with even degree points.
\end{remark}

\begin{say}[Proof of (\ref{desc.meth.say}.1)]

We could just refer to (\ref{2-to-mpts.say}) or to 
\cite[Thm.2]{k-sz}, but I rather explain how to
prove the 2 point case 
using (\ref{2pt.curve.leff.thm})  and the above descent method.

Fix $c,c'\in C(\FF_q)$.
By (\ref{2pt.curve.leff.thm}), there is an open subset
$Y^0\subset X\times X$ such that the following holds

\begin{enumerate}
\item[($*$)]  If $\FF_{q^m}\supset \FF_q$ is large enough
then for every $(x,x')\in Y^0(\FF_{q^m})$ there is
 an $\FF_{q^m}$-map  $\Psi:C\to X$ such that
$\Psi(c)=x$ and $\Psi(c')=x'$.
\end{enumerate}

Assume now that we have
any  $x,x'\in  X(\FF_{q^m})$.
If we can choose the lines $L$ through $x$ and $L'$ through $x'$
such that 
$(s(x), s(x'))\in Y^0$, then the descent method produces
the required extension  $\Psi:C\to X$ over $\FF_{q^m}$.

By the Lang-Weil estimates, $Y^0(\FF_{q^m})$ has about $q^{2nm}$ points.
If, for a line $L$ through $x$, one of the other 
two points of $X\cap L$ is in $\FF_{q^m}$ then so is the other point.
Thus we have about $\frac12 q^{nm}$ lines where
$s(x), s'(x)$ are in $X(\FF_{q^m})$. Accounting for the lines tangent to
$X$ gives a contribution $O(q^{(n-1)m})$. Thus about $\frac14$
of all line pairs $(L,L')$ work for us.
\end{say}
\medskip

The proof of (\ref{desc.meth.say}.2) is an elaboration of
the above line and point counting argument.

\begin{lemma}\label{lem:cubic_conjugate_points}
Let $X\subset \PP^{n+1}$ be a normal cubic hypersurface and
$p \in X(\FF_q)$ a smooth point. Assume that $n\geq 1$ and $q\geq 8$.
Then 
\begin{enumerate}
\item either there is a line defined over
  $\FF_q$  through $p$ but not contained in $X$ that intersects $X$ in two 
further smooth points $s, s'
  \in X(\FF_{q^2}) \setminus X(\FF_q)$,
\item or projecting $X$ from $p$ gives an inseparable degree 2 map
$X\rto \PP^n$. In this case $q=2^m$ and $X$ is singular.
\end{enumerate}
\end{lemma}

\begin{proof} Start with the case $n=1$. Thus $C:=X$ is plane cubic
which we  allow to be reducible.

Consider first the case when 
 $C=L\cup Q$ a line through $p$ and a smooth conic $Q$.
There are $q+1$ $\FF_q$-lines through $p$, one is $L$ and at most 
2 of them are tangent
to $Q$, unless projecting $Q$ from $p$ is purely inseparable. 
If all the remaining $q-2$ lines intersect $Q$ in
two $\FF_q$-points, then $Q$ has $2+2(q-2)=2q-2$ points in
$\FF_q$. This is impossible for $q>3$.
In all other reducible cases, $C$ contains a line not passing through $p$.
(Since $C$ is smooth at $p$, $C$ can not consist of 3 lines 
passing through $p$.)

Assume next that $C$ is irreducible and smooth. 
If projection from $p$ is separable, then at most 4 lines through $p$
are tangent to $C$ away from $p$ and one is tangent at $p$. 
If all the remaining $q-4$ lines intersect $C$ in
two $\FF_q$-points, then $C$ has $5+2(q-4)=2q-3$ points in
$\FF_q$. For $q\geq 8$ this contradicts the Hasse-Weil estimate
$\#C(\FF_q)\leq q+1+2\sqrt{q}$.
The singular case works out even better.

Now to the general case. Assume that in affine coordinates $p$ is the origin
and write the equation as
$$
L(x_1,\dots,x_{n+1})+Q(x_1,\dots,x_{n+1})+C(x_1,\dots,x_{n+1})=0.
$$
Let us show first that there is a line defined over
  $\FF_q$  through $p$ but not contained in $X$ that intersects $X$ in two 
further smooth points $s, s'$.

If the characteristic is $2$, then projection from $p$ is inseparable
iff $Q\equiv 0$. If $Q$ is not identically zero, then
for $q\geq 3$ there are $a_1,\dots,a_{n+1}\in \FF_q$
such that $(L\cdot Q)(a_1,\dots,a_{n+1})\neq 0$. The corresponding line
intersects $X$ in 2 further distinct points, both necessarily smooth.

If the characteristic is $\neq 2$, then the line
corresponding to $a_1,\dots,a_{n+1}\in \FF_q$ has a double
intersection iff the discriminant $Q^2-4LC$ vanishes.
Note that $Q^2-4LC$ vanishes identically only if $X$ is reducible.
Thus, for $q\geq 5$ there are $a_1,\dots,a_{n+1}\in \FF_q$
such that $(L\cdot (Q^2-4LC))(a_1,\dots,a_{n+1})\neq 0$. 
As before, the corresponding line
intersects $X$ in 2 further distinct points, both necessarily smooth.

It is possible that for this line  $s, s'
  \in X(\FF_{q^2}) \setminus X(\FF_q)$ and we are done.
If not then  $s, s' \in  X(\FF_q)$. We can choose
the line to be $(x_1=\cdots=x_n=0)$ and write 
$s=(0,\dots,0,s_{n+1})$ and $s'=(0,\dots,0, s'_{n+1})$.
 Our aim now is to
intersect $X$ with the planes
$$
P(a_1,\dots,a_{n}):=\langle (0,\dots,0,1), (a_1,\dots,a_{n},0)\rangle
$$
 for various $a_1,\dots,a_{n}\in \FF_q$
and show that for one of them the intersection does not contain
 a line not passing through $p$. Then the curve case discussed above
finishes the proof.

Set $x'_{n+1}=x_{n+1}-s_{n+1}$. At $s$ the equation of $X$ is
$$
L_s(x_1,\dots,x'_{n+1})+Q_s(x_1,\dots,x'_{n+1})+
C_s(x_1,\dots,x'_{n+1})=0.
$$
Since $X$ is irreducible, $L_s$ does not divide either $Q_s$ or $C_s$.
$L_s$ contains $x'_{n+1}$ with nonzero coefficient
since the vertical line has intersection number 1 with $X$.
We can use $L_s$ to eliminate $x'_{n+1}$ from $Q_s$ and $C_s$.
As we saw, one of these is nonzero, let it be
$B_s(x_1,\dots,x_n)$. 
Similarly, at $s'$ we get $B'_s(x_1,\dots,x_n)$.

If $X\cap P(a_1,\dots,a_{n})$ contains a line through $s$ (resp.\ $s'$)
then $B_s(a_1,\dots,a_{n})=0$ (resp.\ $B'_s(a_1,\dots,a_{n})=0$).
Thus we have the required $(a_1,\dots,a_{n})$, unless
$B_s\cdot B'_s$ is identically zero on
$\PP^{n-1}(\FF_q)$. This happens only for $q\leq 5$.
\end{proof}

\begin{exercise} Let $H(x_1,\dots, x_n)$
be a homogeneous polynomial of degree $d$.
If $H$ vanishes on $\FF_q^n$ and $q\geq d$
then $H$ is identically zero.
\end{exercise}

\begin{exercise} Set $F(x_0,\dots,x_m)=\sum_{i\neq j} x_i^{2^n}x_j$.
Show that $F$ vanishes on $\PP^m(\FF_{2^n})$ and  for $m$ odd
it defines a smooth  hypersurface.
\end{exercise}

\bibliographystyle{alpha}
\bibliography{kollar}

\vskip1cm

\noindent Princeton University, Princeton NJ 08544-1000

\begin{verbatim}kollar@math.princeton.edu\end{verbatim}

\end{document}